\documentclass[11pt]{article}
\usepackage[top=1in,bottom=1in,right=1in,left=1in]{geometry}
\usepackage{amssymb,amsmath,amsthm}
\usepackage{pgfplots}
\usepackage{pgfplotstable}
\usepackage{arydshln}
\usepackage{verbatim}
\usepackage{titlesec}
\usepackage{subfigure}
\usepackage[linesnumbered,ruled]{algorithm2e}
\usepackage{bm}
\usepackage{pdfpages}
\usepackage{enumerate} 
\definecolor{forestgreen}{rgb}{0.13, 0.55, 0.13}
\usepackage[colorlinks, linkcolor=blue,citecolor=forestgreen,urlcolor = black, bookmarks=true]{hyperref}
\usepackage[T1]{fontenc}
\newcommand{\Mod}[1]{\ (\mathrm{mod}\ #1)}
 
\usepackage[nameinlink, noabbrev, capitalize]{cleveref}
\usepackage{tikz}
\usepackage{times}

\pgfplotsset{compat=1.18}
\newcommand\blfootnote[1]{%
  \begingroup
  \renewcommand\thefootnote{}\footnote{#1}%
  \addtocounter{footnote}{-1}%
  \endgroup
}

\crefname{equation}{}{}
\crefrangeformat{equation}{(#3#1#4) --~(#5#2#6)}
\crefmultiformat{equation}{(#2#1#3)}{, (#2#1#3)}{, (#2#1#3)}{, (#2#1#3)}
\crefname{lem}{Lemma}{Lemmas}
\crefname{section}{Section}{Sections}
\crefname{subsubsubsection}{Section}{Sections}
\crefname{rem}{Remark}{Remarks}
\crefname{figure}{Figure}{Figures}
\crefname{table}{Table}{Tables}
\Crefname{lem}{Lemma}{Lemmas}
\crefname{thm}{Theorem}{Theorems}
\Crefname{thm}{Theorem}{Theorems}

\usepackage{commath}

\newtheorem{theorem}{Theorem}[section]
\newtheorem{lemma}{Lemma}[section]

\newtheorem{corollary}{Corollary}[section]

\newtheorem{definition}{Definition}[section]

\title{Comparison Theorems for the Mixing Times of Systematic and Random Scan Dynamics\blfootnote{J.G. is supported by Vannevar Bush Faculty Fellowship ONR-N00014-20-1-2826 and Simons Investigator Award 622132. E.M. is supported in part by Vannevar Bush Faculty Fellowship ONR-N00014-20-1-2826, Simons Investigator Award 622132, and Simons-NSF DMS-2031883.}}
\author{Jason Gaitonde\\
Massachusetts Institute of Technology\\
\texttt{gaitonde@mit.edu}
\and Elchanan Mossel\\
Massachusetts Institute of Technology\\
\texttt{elmos@mit.edu}}

\begin{document}

\maketitle
\thispagestyle{empty}

\begin{abstract}
    A popular method for sampling from high-dimensional distributions is the \emph{Gibbs sampler}, which iteratively resamples sites from the conditional distribution of the desired measure given the values of the other coordinates. It is natural to ask to what extent does the order of site updates matter in the mixing time? Two  natural choices are (i) standard, or \emph{random scan}, Glauber dynamics where the updated variable is chosen uniformly at random, and (ii) the \emph{systematic scan} dynamics where variables are updated in a fixed, cyclic order. We first show that for systems of dimension $n$, one round of the systematic scan dynamics has spectral gap at most a factor of order $n$ worse than the corresponding spectral gap of a single step of Glauber dynamics, tightening existing bounds in the literature by He, et al. [NeurIPS '16] and Chlebicka,  Łatuszy\'nski, and Miasodejow [Ann. Appl. Probab. '25]. The corresponding bound on mixing times is sharp even for simple spin systems by an explicit example of Roberts and Rosenthal [Int. J. Statist. Prob. '15]. We complement this with a converse statement: if all, or even just one scan order rapidly mixes, the Glauber dynamics has a polynomially related mixing time, resolving a question of Chlebicka,  Łatuszy\'nski, and Miasodejow. 
\end{abstract}

\newpage
\setcounter{page}{1}
\section{Introduction}
An important task in statistics, computer science, and machine learning is sampling from high-dimensional distributions. However, the ``curse of dimensionality'' can make this problem provably computationally hard in general~\cite{DBLP:conf/focs/Sly10,DBLP:conf/focs/SlyS12,DBLP:journals/corr/abs-2407-07645}. Even in settings where sampling may be possible, it is often the case that it is easy to determine the probability of a configuration up to a universal normalizing constant (also called the \emph{partition function}), but the exact computation of this normalization requires summing over (at least) exponentially many configurations in the dimension. Therefore, a popular option in both theory and practice is to run a suitable ergodic Markov chain on the state space that is easy to implement and whose stationary measure is the desired distribution. 

Of these methods, an extremely popular approach for sampling from high-dimensional distributions is the well-known \emph{Gibbs sampler} introduced by Geman and Geman~\cite{geman1984stochastic} inspired by ideas from statistical physics. The Gibbs sampler is particularly convenient since it is usually quite efficient to implement. 
Concretely, the general form of these procedures is as follows: for some target distribution $\pi$ on a product space $\mathcal{X}:=\mathcal{X}_1\times\ldots\times \mathcal{X}_n$, consider any starting configuration $X^0\in \mathcal{X}$. At each time $t\geq 1$, the sampler chooses a coordinate $i_t\in [n]$ according to some predetermined (possibly randomized) rule and rerandomizes the $i_t$'th coordinate of $X^{t-1}$ according to the distribution of $\pi$ on $\mathcal{X}_{i_t}$ conditioned on the remaining coordinates agreeing with $X^{t-1}$. All other coordinates remain the same, and these conditional distributions are often computationally straightforward to compute. It is not difficult to show that $\pi$ is stationary with respect to these dynamics for any choice of indices $i_t$, and thus, the overall complexity of the sampling method is governed by the number of steps for which it needs to run to be close to stationarity.

Two natural ways to choose the coordinate $i_t$ to update at time $t$ are (i) standard (\emph{random scan}) Glauber dynamics, where $i_t\in [n]$ is chosen uniformly at random, and (ii) the \emph{systematic scan}, where $i_t = \sigma(t\Mod n)$ for some fixed choice of permutation $\sigma:[n]\to [n]$ so that the choice of sites to update cycles around the coordinates in some fixed order. For any choice of updates, the crucial quantity of interest is the \emph{mixing time} to stationarity: how many steps of the Gibbs sampler are needed until the distribution at time $T$ is close to $\pi$ in, say, total variation distance? Ideally, the process mixes in time $\mathsf{poly}(n)$ so that the overall sampling complexity scales only polynomially in the dimension.

The systematic scan is often used in practice due to several implementation advantages over Glauber dynamics~\cite{DBLP:conf/approx/DyerGJ06,DBLP:conf/sigmod/ZhangR13}. However, they often lack strong analytic guarantees on the mixing times due to the lack of reversibility. General comparisons between the two chains have been conjectured for several decades~\cite{diaconis,levin2017markov}, but it is known that the site ordering can in fact substantially affect the mixing time~\cite{roberts2015surprising,DBLP:conf/nips/HeSMR16}: simple examples show that either sampler can dominate the other by $\mathsf{poly}(n)$ factors in mixing time, different scan orders can have a $\mathsf{poly}(n)$ factor effect on the mixing time, and the random scan can lie asymptotically strictly between the fastest and slowest systematic scans. To the best of our knowledge, a precise \emph{worst-case} comparison of the spectral properties or mixing times of the systematic scan in terms of Glauber has remained open despite significant work comparing them in both special cases and general settings~\cite{amit,dyer_colors,diaconis_ram,DBLP:conf/approx/BlancaZ23,simon,DBLP:conf/approx/DyerGJ06,fishman,DBLP:conf/nips/HeSMR16,chlebicka}. In the other direction, establishing any explicit polynomial upper bound on the mixing time of Glauber dynamics in terms of the mixing times of the standard systematic scan seems to have remained open as well, appearing as a recent question of Chlebicka,  Łatuszy\'nski, and Miasodejow~\cite{chlebicka}.

In this work, we resolve both of these fundamental questions. In particular, we show the following:
\begin{enumerate}
    \item We \emph{completely resolve} the worst-case spectral gap of any systematic scan in terms of the spectral gap of Glauber dynamics: if the Glauber dynamics has spectral gap at least $\gamma$, then the full $n$ steps of the systematic scan has spectral gap at least $\Omega(\gamma/n)$, and this factor is sharp. To the best of our knowledge, this is even the first such comparison result where the ratio between the gaps 
    is independent of the original spectral gap, which can be exponentially small in many known examples where $\gamma = e^{-\Omega(n)}$. Since the scan has $n$ steps, this implies via standard relationships between spectral gaps and mixing times that the number of steps of the scan to mix is at most a factor of $n^2\log(1/\pi_{\mathsf{min}})$ larger than that of Glauber.\footnote{The logarithmic factor is inherent to spectral methods to analyzing the mixing times; it is an interesting question whether this factor can removed by comparing the modified log-Sobolev constants of these chains.} We thus prove that an explicit example of Roberts and Rosenthal~\cite{roberts2015surprising} (namely, the hardcore model on the complete graph) constitutes the provably worst-case instance of mixing times ratio between these chains (up to a constant).  
    \item Moreover, we establish a polynomial upper bound on the spectral gap of the Glauber dynamics in terms of the spectral gap of systematic scans. By leveraging a recent spectral gap inequality of Chatterjee~\cite{chatterjee2023spectral} on the singular values of the Laplacian, this in fact yields a polynomial upper bound in terms of mixing times as well due to the special orthogonality properties of the scan. To the best of our knowledge, this appears to be the first such bound and resolves a question (Open Problem 3.7) of Chlebicka,  Łatuszy\'nski, and Miasodejow~\cite{chlebicka}.
\end{enumerate}
As we describe further below, our analysis naturally extends to any sufficiently pseudorandom sequence of site re-samplings, which may have statistical applications in natural settings where Gibbs sampling is viewed as an \emph{exogenous} process describing a system's evolution rather than as a sampling method by a practitioner. An additional bonus of our approach is that it is much more direct, but at the same time sharper, than the significantly more involved functional-analytic approach of Chlebicka,  Łatuszy\'nski, and Miasodejow~\cite{chlebicka}. Their result, which we describe in more detail in \Cref{sec:related}, obtained the previous best, but suboptimal, bound on the spectral gap of the systematic scan that we are aware of. By contrast, our results rely on conceptually natural linear-algebraic and probabilistic arguments that elementary, flexible, and can lead to optimal bounds.

\subsection{Background and Our Results}
We now describe the background and our results in more detail. By now, a variety of advanced methods have been introduced to bound the mixing time of important Markov chains (see, e.g. \cite{levin2017markov,tetali} for textbook treatments). For high-dimensional distributions on product spaces, mixing time analyses are typically studied for the Glauber dynamics since this Markov chain enjoys several analytical advantages. For one, the Markov chain is reversible with respect to $\pi$ so that the mixing properties can be analyzed by well-understood spectral methods. In particular, the mixing time connects naturally to linear-algebraic machinery via the spectral theorem, making these dynamics significantly more amenable to a rigorous study in most applications. Moreover, the Glauber dynamics treats sites symmetrically and is the natural tensorized analogue of the transition operators that operates independently on each coordinate, as would be the case for product distributions. 

By contrast, the transition operator associated with the systematic scan is importantly \emph{not} typically reversible once $n>2$ except in simple settings, like when $\pi$ is a product measure. Even though the transitions are a product of self-adjoint operators, the lack of commutativity makes the analysis highly nontrivial. In fact, this issue is not just a feature of existing technical machinery: the actual site ordering can indeed significantly affect the mixing time since the application of the transitions varies deterministically depending on the iteration~\cite{roberts2015surprising,DBLP:conf/nips/HeSMR16}. Nonetheless, from the computational side, the fixed and sequential nature of the systematic scan is among several practical considerations that can affect runtime in real-world implementations~\cite{DBLP:conf/sigmod/ZhangR13}. No additional randomness is needed in the sequence of updated sites and moreover, one can exploit locality among adjacent sites in the updating process as well as parallelization across distant sites, which can yield practical gains given hardware considerations. For these reasons, experimental evaluations are often done using the systematic scan~\cite{DBLP:conf/approx/DyerGJ06}.

As a result, it is important to understand the relative performance between these two implementations of the Gibbs sampler. More generally, to what extent does the choice of site updates affect the mixing time of important high-dimensional distributions one wishes to sample from? From an analytical perspective, it may be significantly easier to prove mixing time bounds for Glauber dynamics for the reasons above; but then to what extent can these guarantees extend to the systematic scan? Obtaining (hopefully sharp) \emph{black-box} guarantees of this type provide concrete performance guarantees for the systematic scan.

It had originally been conjectured by Diaconis~\cite{diaconis} and Levin and Peres~\cite{levin2017markov} that, at least for important classes of distributions, the mixing time of the Glauber dynamics should be comparable to the mixing time of the systematic scan measured per step. That is, the total number of steps of either Gibbs sampler to reach some $\varepsilon$ distance to $\pi$ in total variation should be within $O(\log(n))$ factors of each other that are necessary for product measures.  However, such a result cannot hold in general: Roberts and Rosenthal~\cite{roberts2015surprising} as well as He, de Sa, Mitliagkas, and R\'e~\cite{DBLP:conf/nips/HeSMR16} construct several examples of models demonstrating that the relationship between mixing times and scan order is quite subtle: either sampler can dominate the other by $\mathsf{poly}(n)$ factors in mixing time, different scan orders can have a $\mathsf{poly}(n)$ factor effect on the mixing time, and the random scan can lie asymptotically strictly between the fastest and slowest scans. A comprehensive and rigorous analysis of conditions where one scan should perform as well as the other remains elusive, and even the \emph{worst-case} comparison of these Gibbs samplers appears not to have been resolved to our knowledge (see \Cref{sec:related} for existing bounds in the literature).

In this paper we make significant 
qualitative and quantitative progress  
in 
 resolving the worst-case ratios
 of the mixing times of these two important chains. We first prove the following result on mixing times. To state the result, let $t^{\mathsf{GD}}_{\mathrm{mix}}(\varepsilon)$ denote the time for Glauber dynamics to be within $\varepsilon$ of $\pi$ in total variation distance from any starting configuration, and let $t^{\mathsf{SS}(\sigma)}_{\mathrm{mix}}(\varepsilon)$ denote the time for Glauber dynamics to be within $\varepsilon$ of $\pi$ in total variation distance from any starting configuration, where the scan is done in order $\sigma:[n]\to [n]$ (see \Cref{sec:preliminaries} for more formal definitions).

\begin{theorem}
\label{thm:mixing}
    There exists an absolute constant $C>0$ such that the following holds for any $0<\varepsilon\leq 1$. Let $\mathcal{X}=\mathcal{X}_1\times \ldots\times \mathcal{X}_n$ be a product of finite state spaces\footnote{We expect that our structural results can be applied to more general state spaces, but we restrict to the finite case to avoid measure-theoretic complications.} and let $\pi$ be any distribution on $\mathcal{X}$. Then for any permutation $\sigma:[n]\to [n]$,
    \begin{equation*}
        t^{\mathsf{SS}(\sigma)}_{\mathrm{mix}}(\varepsilon)\leq Cnt^{\mathsf{GD}}_{\mathrm{mix}}(1/4)\log\left(\frac{1}{\varepsilon\pi_{\mathrm{min}}}\right),
    \end{equation*}
    where $\pi_{\mathrm{min}}:=\min_{\bm{x}\in \mathcal{X}:\pi(\bm{x})\neq 0} \pi(\bm{x})$. This bound is tight up to the logarithmic factor.
\end{theorem}

Since the systematic scan consists of $n$ site updates, \Cref{thm:mixing} asserts that in terms of the number of updates, the mixing time of the systematic scan is no more than a factor $n^2$ worse than that of Glauber dynamics up to the logarithmic factor inherent to spectral analyses of mixing. In natural discrete systems, the minimum probability satisfies $\pi_{\min}= 2^{-\Theta(n)}$, in which case \Cref{thm:mixing} asserts that the mixing time is at most a factor of $n^3$ worse. To the best of our knowledge, this is the first bound showing that the ratio of mixing times is at most polynomial, even in settings where the mixing times are themselves superpolynomially large.

Importantly, \Cref{thm:mixing} is not improvable in the worst-case even in simple (anti-ferromagnetic) spin systems as witnessed by an example of Roberts and Rosenthal~\cite{roberts2015surprising}. For completeness, we provide a simple, self-contained proof of this example following their reasoning in \Cref{sec:optimality}. Thus, \Cref{thm:mixing} constitutes the sharpest possible upper bound on the mixing time of the systematic scan that can hold unconditionally for all systematic scans for Gibbs samplers in terms of the mixing time of Glauber dynamics, at least up to the unavoidable logarithmic factor.

We in fact show that our spectral gap bound underlying \Cref{thm:mixing} holds for almost every \emph{pseudorandom} sequence of sites to update. When translated to the Glauber dynamics, or most similar dynamics, our argument yields an efficient \emph{certification} procedure that a process has mixed, \emph{even conditioned on the sites that update}, that holds with $1-o(1)$ probability for random uniform sequences. Understanding which \emph{deterministic} sequences of site updates mix is of importance in applications beyond sampling, like in statistical applications where one does not have perfect knowledge of how sites are chosen to update. For instance, the Glauber dynamics are also considered in economics as ``noisy best-response dynamics'' for decentralized game dynamics \cite{kandori,BLUME1993387,young} as well as in statistical physics as a model for a particle system to converge towards thermal equilibrium. In these systems, one views the dynamics as arising exogenously and instead wishes to learn some macroscopic properties or statistics of typical samples from $\pi$. However, it may not be possible to know the precise (possibly stochastic) process that generates site updates; can the statistician nonetheless conclude that after some \emph{predictable} period of time, the observed configuration is close to the true stationary measure without knowledge of the precise randomness in the \emph{sequence} of observed site updates? This amounts to asserting that the configuration has conditional law close to $\pi$ even given the (independently chosen) sequence of updated sites.

Our argument (which we discuss in more detail in \Cref{sec:overview_tech}) for establishing \Cref{thm:mixing} and these applications turns out to be significantly more flexible than the analysis of Chlebicka,  Łatuszy\'nski, and Miasodejow~\cite{chlebicka} in that it can be more or less reversed up to polynomial factors. We further show that if \emph{any} systematic scan has a spectral gap, then Glauber dynamics has a polynomially related spectral gap, which therefore implies a mixing bound by standard arguments. This spectral gap comparison resolves Open Problem 3.7 of Chlebicka,  Łatuszy\'nski, and Miasodejow~\cite{chlebicka} with the following consequence:

\begin{theorem}
\label{thm:mixing_2}
         Let $\mathcal{X}=\mathcal{X}_1\times \ldots\times \mathcal{X}_n$ be a product of finite state spaces and let $\pi$ be any distribution on $\mathcal{X}$. Then for any $0<\varepsilon\leq 1$ and any permutation of $[n]$, $\sigma$,  
    \begin{equation*}
        t^{\mathsf{GD}}_{\mathrm{mix}}(\varepsilon)\leq \mathsf{poly}\left(n,t^{\mathsf{SS}(\sigma)}_{\mathrm{mix}}(1/10)\right)\log\left(\frac{1}{\varepsilon\pi_{\mathrm{min}}}\right).
    \end{equation*}
\end{theorem}

We provide the precise quantitative dependencies in \Cref{sec:proof}; the bound improves depending on the assumption made on whether all systematic scan orders $\sigma$ mix fast, or whether there is a single one that mixes fast. While these bounds are significantly weaker than our tight results for the converse direction, \Cref{thm:mixing} and \Cref{thm:mixing_2} together constitute an unconditional polynomial equivalence between the mixing times of the Glauber dynamics and any systematic scan. The mixing time comparison follows from the spectral gap comparison by a recent comparison of the mixing time of non-reversible chains in terms of the minimal nontrivial singular value of the Laplacian by Chatterjee~\cite{chatterjee2023spectral}. While this value is usually not related to the operator norm for general non-reversible Markov chains, the projection property of the Gibbs sampler will enable us to relate these quantities up to polynomial factors.\\

\subsection{Overview of Techniques}
\label{sec:overview_tech}
We now turn to giving a high-level overview of our main techniques towards establishing \Cref{thm:mixing} and \Cref{thm:mixing_2}. All precise definitions are given in \Cref{sec:preliminaries}.

\noindent \textbf{From Glauber Dynamics to  Scans.}
The proof of \Cref{thm:mixing} and the related applications for pseudorandom sequences is based on linear-algebraic arguments that bound the \emph{operator norms of appropriate matrices}. The connection between spectral gaps and mixing times is fairly standard, so we only briefly remark on this connection here and defer the standard background and proofs to \Cref{sec:preliminaries}. 

At a high level, it is well-known that the Glauber dynamics having relaxation time to stationarity at most $1/\delta$ is the same as asserting that for any function $f\in L^2(\mathcal{X},\pi)$ with $\text{Var}_{\pi}(f)=1$ and $\mathbb{E}_{\pi}[f]=0$,
\begin{equation*}
    \left\|\frac{1}{n}\sum_{i=1}^n \mathsf{P}_i f\right\|_{2,\pi}\leq 1-\delta,
\end{equation*}
for appropriate linear operators $\mathsf{P}_i$ that are self-adjoint with respect to the natural inner product associated to $ L^2(\mathcal{X},\pi)$. Here, the norm is the induced operator norm on this space. When acting on distributions on the left, each $\mathsf{P}_i$ corresponds to transitioning by rerandomizing in the $i$th site as in the Gibbs sampler. These operators act on functions $f\in L^2(\mathcal{X},\pi)$ on the right by taking the conditional expectation of $f$ after rerandomizing the $i$th site given the remaining coordinates, which is well-known to be an orthogonal projection in $L^2(\mathcal{X},\pi)$. 

Our first main technical result establishes a suitable operator norm bound on the product of transition matrices associated to any systematic scan. Concretely, we prove the following abstract statement:

\begin{theorem}
[\Cref{thm:gen_gap}, restated]
\label{thm:gap_intro}
        Let $\mathcal{H}$ be a Hilbert space and suppose $\mathsf{P}_1,\ldots, \mathsf{P}_n$ are orthogonal projections with respect to the inner product of $\mathcal{H}$ satisfying
        \begin{equation}
    \left\|\frac{1}{n}\sum_{i=1}^n \mathsf{P}_i\right\|_{\mathsf{op}}\leq 1-\delta.
\end{equation} 
Let $\mathsf{P}_{i_L}\ldots \mathsf{P}_{i_1}$ be any product of the orthogonal projection operators satisfying the following condition. For each $j\in [n]$, suppose there exists a minimal index $k_j\leq L$ satisfying $i_{k_j}=j$. We further assume these indices satisfy:
\begin{equation}
\label{eq:order}
        1=k_1\leq k_2\leq\ldots\leq k_n=L.
\end{equation}
Then, it holds that
\begin{equation*}
    \|\mathsf{P}_{i_L}\ldots \mathsf{P}_{i_1}\|_{\mathsf{op}}^2\leq 1-\frac{n\delta}{8\sum_{j=1}^n k_j}.
\end{equation*}
\end{theorem}

In words, \Cref{thm:gap_intro} asserts that if the arithmetic mean of orthogonal projection operators satisfies a spectral gap, then the spectral gap of any product of them can be bounded by an explicit quantity that depends only on the first appearances of each distinct projection; this result may be of independent interest. As we discuss in \Cref{sec:optimality}, this result is completely sharp up to the constant in two respects: first, it is necessary that the operators be orthogonal projections (or else the bound can even be \emph{exponentially large}!), and second, there are explicit examples of sequences of orthogonal projections~\cite{DBLP:journals/jmlr/RechtR12} even in $\mathbb{R}^2$ showing the conclusion is tight up to the precise constant.

Returning to the setting of the systematic scan, we have $k_j=j$ by definition, in which case the corresponding sum is $\Theta(n^2)$ (\Cref{cor:spectral_scan}). The same can shown with high probability for a sequence of indices drawn uniformly at random from $[n]$ (c.f. \Cref{thm:pseudo}), and hence a similar contraction holds for any sufficiently pseudorandom sequence. In both cases, the conclusion is that the corresponding spectral gap of the sequence of operators is at least $\Theta(\delta/n)$. By standard arguments relating mixing times and operator norms even for non-reversible Markov chains which we recall in \Cref{sec:preliminaries}, we obtain 
\Cref{thm:mixing}. Again, we establish for the first time that a simple example of Roberts and Rosenthal~\cite{roberts2015surprising} yields a tight example for the worst-case mixing comparison. 

To establish \Cref{thm:gap_intro}, we use the following reasoning. For each $\ell\geq 1$, we define the operators 
\begin{equation*}
    \mathsf{Q}_{\ell}:=\mathsf{P}_{i_{\ell}}\cdots \mathsf{P}_{i_1}.
\end{equation*}
Our goal is to give an operator norm bound on $\mathsf{Q}_L$ and we do so by precisely tracking the spectral gap of each of the intermediate operators $\mathsf{Q}_{\ell}$. The key observation is that for any vector $f\in \mathcal{H}$, we can write the \emph{orthogonal decomposition}
\begin{equation*}
        \mathsf{Q}_{{\ell}}f = \mathsf{P}_{i_{\ell+1}}( \mathsf{Q}_{{\ell}}f)+(\mathsf{Q}_{{\ell}}f-P_{i_{\ell+1}}\mathsf{Q}_{{\ell}} f)= \mathsf{Q}_{{\ell+1}}f+(\mathsf{Q}_{{\ell}}f-\mathsf{Q}_{{\ell+1}} f)
    \end{equation*}
We exploit this orthogonality with a precise charging argument to bound the norm of the difference $\mathsf{P}_jf-f$ in terms of the individual movements $\|\mathsf{Q}_{\ell}f-\mathsf{Q}_{\ell-1}f\|$ up to $\ell=k_j$. The intuition is that the spectral gap of Glauber, being an average of such differences, can be charged to the individual movements when we iteratively apply each projection operator. More precisely, if the first application of $\mathsf{P}_j$ to $f$ in the iterated projections does \emph{not} have a similar effect as it does for the average projection, then this must be because the preceding projection operators noticeably moved $f$. But in this case, this increases the spectral gap of the product, since any movement from an orthogonal projection decreases the norm! Making this win-win analysis completely precise follows essentially from the Pythagorean identity that \emph{squared norms} are exact when written for an orthogonal decomposition. Since the average of operators is assumed to have a spectral gap, we can then ensure that the $\mathsf{Q}_{\ell}$ operators also make progress in decreasing the norm of $f$. In light of \Cref{sec:optimality}, our reasoning and accounting in this charging argument is completely tight.

\noindent\textbf{From Scans to Glauber.} To resolve the polynomial equivalence of the two mixing times, we then turn to establishing \Cref{thm:mixing_2}.
A first important subtlety is that for non-reversible chains, the connection between mixing times and spectral gaps is not exact; a spectral gap \emph{upper bounds} the mixing time, but a lower bound typically only follows from a large norm \emph{eigenvalue}. In particular, the fact that a Markov chain mixes fast does not in itself imply a good bound on the spectral gap. As a result, it is not immediately clear how to employ similar reasoning to bound the mixing time of Glauber, since this is equivalent to a spectral gap, under just the assumption a scan is fast mixing. 

Our first simple contribution, that may be of independent interest, is noting that for the systematic scan, there \emph{is} such a lower bound. In particular, we exploit  a recent connection of Chatterjee~\cite{chatterjee2023spectral} that establishes lower bounds on mixing times in terms of spectral gaps of the \emph{Laplacian} of the scan. For a general Markov chain $\mathsf{P}$, the Laplacian is defined as $\mathsf{I}-\mathsf{P}$, and there is not usually a direct comparison between the singular values of the Laplacian and the original chain when nonreversible. However, we show in \Cref{lem:gibbs_comp} that there is such a comparison for the scan \emph{precisely} due to the fact each individual operator in the systematic scan is an orthogonal projection.

With this in order, we can then aim to establish a converse on bounding the mixing time of Glauber via one or all scans in terms of spectral gaps. The key technical claim we show is the following:

\begin{theorem}[\Cref{thm:gen_gap_2}, restated]
\label{thm:gap_intro_2}
        Let $\mathcal{H}$ be a Hilbert space and suppose $\mathsf{P}_1,\ldots, \mathsf{P}_n$ are orthogonal projections with respect to the inner product of $\mathcal{H}$. Suppose that
        \begin{equation}
    \left\|\mathsf{P}_n\ldots \mathsf{P}_1\right\|_{\mathsf{op}}\leq 1-\delta.
\end{equation} 
Let $\mathsf{P}_{i_L}\ldots \mathsf{P}_{i_1}$ be any product of these orthogonal projection operators satisfying the following condition. For each $j\in [n]$, suppose there exists a minimal index $k_j\leq L$ satisfying $i_{k_j}=j$. We further assume these indices satisfy:
\begin{equation}
\label{eq:order_2}
        1=k_1\leq k_2\leq\ldots\leq k_n=L.
\end{equation}
Then, it holds that
\begin{equation*}
    \|\mathsf{P}_{i_L}\ldots \mathsf{P}_{i_1}\|_{\mathsf{op}}\leq 1-\frac{\delta^2}{8(L-n+1)}.
\end{equation*}
\end{theorem}
In words, we show that if a fixed sequence of the orthogonal projections has a spectral gap, then the same will be true of \emph{any supersequence} where the spacing between the first appearances of the operators \emph{in order} is not too large. This statement and conclusions are qualitatively similar to \Cref{thm:gap_intro}, but loses in the square due to less favorable orthogonality arguments that are lossy. To prove \Cref{thm:gap_intro_2}, we employ a similar careful win-win analysis; either the first application of a projection $\mathsf{P}_j$ affects a given vector in similar ways, or the loss can be \emph{charged} to the additional projections in the supersequence. Balancing these contributions properly leads to the final bound.

To apply this result to bounding the spectral gap of Glauber, we use the following key observation: since the Glauber dynamics is reversible, the operator is self-adjoint, and hence 
\begin{equation*}
        \left\|\frac{1}{n}\sum_{i=1}^n \mathsf{P}_i\right\|^L =\left\|\left(\frac{1}{n}\sum_{i=1}^n \mathsf{P}_i\right)^L\right\|.
    \end{equation*}
    Therefore, it suffices to bound the spectral gap for a suitable power. The key idea is that taking $L\approx n\log(n)$ or $L\approx n^2\log(n)$, expanding the power leads to sums of products of the projections that with very high probability contain at least one scan or some \emph{fixed} scan ordering as an embedded subsequence. In particular, we may apply \Cref{thm:gap_intro_2} to assert that most of these terms have an explicit spectral gap, while we may trivially bound the rest of the terms by $1$. Since the probability of a random sequence of $L$ projections containing a scan decays much much faster than $L$, the loss in the denominator of \Cref{thm:gap_intro_2} is much less than the error probability of not containing a good scan. Setting parameters precisely leads to the polynomial equivalence of \Cref{thm:mixing_2}, thus resolving the open problem of Chlebicka, Łatuszy\'nski, and Miasodejow~\cite{chlebicka}.

\noindent\textbf{Organization.} In \Cref{sec:related}, we discuss related bounds and comparisons in the literature between Glauber dynamics and the systematic scan; to our knowledge, our spectral gap comparison subsumes all such existing bounds and cannot be improved in light of known lower bounds. In \Cref{sec:preliminaries}, we include for completeness standard results on the relation between the spectral gap and mixing times, including for non-reversible chains, for our main application. In \Cref{sec:proof}, we prove the requisite linear-algebraic facts that yield our mixing time comparisons. Finally, we show in \Cref{sec:optimality} that \Cref{thm:mixing} and \Cref{thm:gap_intro} cannot be significantly improved via simple, existing examples.

\subsection{Related Work}
\label{sec:related}
As described above, a general equivalence between the systematic scan and Glauber dynamics has been conjectured to hold in certain settings~\cite{diaconis,levin2017markov}, but the relationship between the two has remained unclear. While the systematic scan is often difficult to study due to the unfavorable analytical properties above, some results of this type have been established in certain models. Among others, such results have been carefully studied in the special cases of sampling from multivariate Gaussians~\cite{amit}, colorings~\cite{dyer_colors}, groups~\cite{diaconis_ram}, and monotone spin systems under spectral independence~\cite{DBLP:conf/approx/BlancaZ23}. It has also been long known that under Dobrushin-type conditions, both Glauber dynamics and the systematic scan mix rapidly~\cite{simon,DBLP:conf/approx/DyerGJ06}. We remark that the work of Dyer, Goldberg, and Jerrum~\cite{DBLP:conf/approx/DyerGJ06} shows that if one defines Glauber and the systematic scan as the average or product of more general Markov transitions at each site (i.e. not conditional re-randomizations), then the systematic scan need not even be ergodic even if Glauber rapidly mixes. Therefore, the specific property of Gibbs samplers as orthogonal projections is essential for any spectral or mixing comparison. See also Fishman~\cite{fishman} for an early study of the relative performance of different scan orders.

To our knowledge, our work is the first to both establish a \emph{tight} comparison of the two samplers in one direction as well as a polynomial equivalence of the two samplers in general. The main result of Chlebicka,  Łatuszy\'nski, and Miasodejow~\cite{chlebicka} yields the following suboptimal guarantee: if the Glauber dynamics has spectral gap of order $\gamma$, then a full round of any systematic scan has spectral gap at least $\gamma^2/n^2$, which is quadratically worse than our optimal bound in both components. Consequently, their result cannot establish the (weaker) claim that the mixing time of systematic scans are  at most $\mathsf{poly}(n)$ factors worse of Glauber dynamics in low-temperature systems with superpolynomial mixing times, nor ever prove a mixing time bound better than $\Theta(n^5)$ steps for the systematic scan in any nontrivial system where $\gamma\lesssim 1/n$. At the technical level, their work employs  machinery of Badea, Grivaux, and M\"uller~\cite{badea} on the convergence of alternating projections. Our approach is conceptually cleaner and more direct, yielding the optimal spectral gap. We also provide a qualitative converse that establishes a polynomial relationship in the reverse direction, answering Open Question 3.7 of their work using variations of this technique, emphasizing the applicability of our methods. By contrast, their results (as proven in earlier work of~\cite{roberts_rosenthal}) only establish an exponentially small spectral gap for the converse exploiting the trivial fact that expanding $n$ steps of Glauber includes a single systematic scan (out of all $n^n$ choices of sites) with an assumed spectral gap. We use our more flexible and direct framework for analyzing spectral gaps to establish the first polynomial relationship in this direction.

The earlier work of He, de Sa, Mitliagkas, and R\'e~\cite{DBLP:conf/nips/HeSMR16} provides a litany of examples showing no asymptotically general ordering of scan dynamics can hold, and then provide a comparison result using conductances. However, this argument requires them to compare a \emph{lazy, augmented} systematic scan with the Glauber dynamics, where the augmented state space is enlarged to incorporate the index of the site to be updated at the current state. Their result shows that the mixing time of Glauber can be bounded by the essentially the square of the mixing time of a lazy, augmented scan.  On the other hand, their mixing time bound for the lazy, augmented scan only yields an upper bound in terms of Glauber that is quantitatively suboptimal in terms of additional $\mathsf{poly}(n,\tau_{\mathsf{mix}})$ factors and further scales inversely with the minimum holding probability of any transition, which can be exponentially small even in rapidly mixing models like high-temperature spin glasses.

The special case where $n=2$ can be viewed as alternating projections and admits considerably sharper guarantees. As demonstrated by Roberts and Rosenthal~\cite{roberts2015surprising}, several phenomena that arise are particular to this case and cannot extend to larger $n$, so the techniques and results do not appear to have much bearing for the general problem we consider. At the linear-algebraic level, a version of the matrix AM-GM inequality is known to hold~\cite{BHATIA2000203} when $n=2$; this implies that the spectral gap of the scan is at least as large as that of Glauber. This is intuitively clear for the Gibbs sampler, since there is no benefit from applying a site transition twice in a row. Several more precise results for this special case have been established in the literature, see e.g. \cite{diaconis2010stochastic,andrieu,DBLP:conf/aistats/0001KZ18,qin} for more discussion.

Understanding the mixing properties of the systematic scan has also been the subject of reason work in the literature on mixing in high-dimensional expanders. Recent work of Alev and Parzanchevski~\cite{alev_1} shows that in extremely good high-dimensional spectral expanders, the spectral gap of the systematic scan can be arbitrarily close to 1 while the spectral gap of $n$ steps of Glauber will remain bounded above by a small constant. This result thus formalizes a sense in which a distribution that is close enough to a product will have a more favorable spectral gap from the scan than from Glauber (i.e. the ``down-up walk" in this literature). A follow-up work by Alev and Rao~\cite{alev_rao} considers ``expanderized up-down walks,'' a higher-dimensional analogue of the ``derandomized square''~\cite{DBLP:conf/approx/RozenmanV05} from the pseudorandomness literature that replaces the full randomness of the site updates of Glauber (which corresponds to the complete graph on $[n]$) with a random walk on a sparse base graph on $[n]$. Thus, one can still obtain sharp mixing times using only $O(1)$ random bits to choose sites at each time, rather than requiring $O(\log n)$ bits.  However, their results do not appear to give any bound on the actual systematic scan since this corresponds to using the directed cycle as the base graph, which has unfavorable spectral properties. 

We remark that from the perspective of derandomizing Glauber dynamics, one may also appeal to seminal pseudorandomness results of Nisan and Zuckerman~\cite{DBLP:journals/jcss/NisanZ96} to randomness-efficiently implement Glauber dynamics on a system with polynomial spectral gap. The key point is that Glauber dynamics can typically be implemented in $O(n)$ space, so their results imply a pseudorandom generator with seed length $O(n)$ that can be used for the implementation of the $\mathsf{poly}(n)$ steps of Glauber with small statistical error. See also the recent work of Feng, Guo, Wang, Wang, and Yin~\cite{derandomizingmcmc} for a complete derandomization of MCMC in certain combinatorial applications.

\section{Preliminaries}
\label{sec:preliminaries}
In this section, we recall standard definitions and results in the theory of Markov chains, Gibbs sampling, and mixing times of finite Markov chains. For a vector $\bm{x}\in \mathcal{X}^n$, we write $\bm{x}_{-i}$ for the restriction to the coordinates outside $i$. We write $\tilde{O}(f(n))$ to denote an upper bound of the form $Cf(n)\log^C(f(n))$ for some universal constant $C>0$. We also write $\mathsf{Bin}(n,p)$ to denote a Binomial random variable with $n$ trials with $p$ success probability and $\mathsf{Geom}(\lambda)$ to denote a Geometric random variable with success probability $\lambda$ (and thus mean $1/\lambda$).
\subsection{Markov Chains, Gibbs Samplers and Linear Algebra}
In this subsection, we provide the standard definitions of  Markov chains and Gibbs samplers. Our main mixing applications will be phrased for finite state spaces in which case these objects reduce to more familiar matrix formulations.
\begin{definition}
    Let $(\mathcal{X},\mathcal{F})$ be a measurable space. A \textbf{Markov kernel} $\mathsf{P}:\mathcal{X}\times \mathcal{F}\to \mathbb{R}$ is a map satisfying
    \begin{enumerate}
        \item For each $x\in \mathcal{X}$, the restriction $F\mapsto \mathsf{P}(x,F)$ forms a probability distribution, and
        \item For each $F\in \mathcal{F}$, the restriction $x\mapsto \mathsf{P}(x,F)$ is a measurable function.
    \end{enumerate}

    We then define the discrete-time Markov chain $X_0,X_1,\ldots$ as follows: let $\mu$ be a probability distribution on $\mathcal{X}$. Then for $t=0$, $X_0\sim \mu$, while $X_t$ has marginal law $\mu \mathsf{P}^t$ and conditional law that recursively satisfies the \textbf{Markov property}:
    \begin{equation*}
        \Pr(X_t\in F\vert X_0,\ldots,X_{t-1})=\mathsf{P}(X_{t-1},F).
    \end{equation*}

    We say that a probability distribution $\pi$ is \textbf{stationary} with respect to $\mathsf{P}$ if $\pi \mathsf{P}=\pi$. The Markov kernel is \textbf{reversible} if $\pi(dx)\mathsf{P}(x,d\pi(y))=\pi(dy)\mathsf{P}(y,d\pi(x))$ for all $x,y\in \mathcal{X}$.
\end{definition}

We recall that for any Markov kernel $\mathsf{P}$, there is a natural dual action on functions. To state this, we define the set
\begin{equation*}
    L^2(\mathcal{X},\pi) = \left\{f:\mathcal{X}\to \mathbb{R}: \int_{\mathcal{X}} \vert f(x)\vert^2d\pi(x)<\infty\right\}.
\end{equation*}
We may then define an inner product and norm with respect to $\pi$ on $L^2(\mathcal{X},\pi)$ via
\begin{gather*}
    \langle f,g\rangle_{\pi} = \int_{\mathcal{X}} f(x)g(x)d\pi(x)\\
    \|f\|_{\pi} = \sqrt{\langle f,f\rangle_{\pi}} = \sqrt{\int_{\mathcal{X}} \vert f(x)\vert^2d\pi(x)}.
\end{gather*}
When there is no ambiguity, we may drop the subscript. Then the transition operator $\mathsf{P}$ naturally acts on functions $f\in L^2(\mathcal{X},\pi)$ via
\begin{equation*}
    \mathsf{P}f(x) = \int_{\mathcal{X}} f(y)\mathsf{P}(x,d\pi(y)) = \mathbb{E}[f(X^1)\vert X^0 = x],
\end{equation*}
where $X^1$ is obtained by applying the Markov kernel to $X^0=x$.
It is well-known that $\mathsf{P}$ is a bounded linear operator that is a contraction in $L^2$. Moreover, it is also well-known that reversibility of $\mathsf{P}$ is equivalent to the operator being self-adjoint i.e. by Fubini's theorem,
\begin{align*}
    \langle f,\mathsf{P}g\rangle &= \int_{\mathcal{X}} f(x)\left(\int_{\mathcal{X}} g(y)\mathsf{P}(x,d\pi(y))\right)d\pi(x)\\
    &=\int_{\mathcal{X}}\int_{\mathcal{X}} f(x)g(y) \mathsf{P}(x,d\pi(y))d\pi(x)\\
    &=  \int_{\mathcal{X}} \left(\int_{\mathcal{X}} f(x)\mathsf{P}(y,d\pi(x))\right) g(y)d\pi(y)\\
    &=\int_{\mathcal{X}} \mathsf{P}f(y) g(y)d\pi(y)\\
    &=\langle \mathsf{P}f,g\rangle.
\end{align*}
Since this holds for all $f,g\in L^2(\mathcal{X},\pi)$, it is possible to show this an equivalence.

We now specialize to the setting of the Gibbs samplers. For each $i\in [n]$, suppose $(\mathcal{X}_i,\mathcal{F}_i)$ is a measure space with some base measure $\mu_i$. Then let $(\mathcal{X},\mathcal{F})=(\mathcal{X}_1\times\ldots\times \mathcal{X}_n,\otimes_{i=1}^n \mathcal{F}_i)$ denote the product space and let $\mu$ be the natural product measure, and suppose $\pi$ is a probability measure on $(\mathcal{X},\mathcal{F})$ that is absolutely continuous with respect to $\mu$. Then for each $i\in [n]$, any set $A\in \mathcal{F}$, and any element $\bm{x}\in \mathcal{X}$, let
\begin{equation*}
    A_{i,\bm{x}} = \{y\in \mathcal{X}_i: (x_1,\ldots,x_{i-1},y,x_{i+1},\ldots,x_n)\in A\},
\end{equation*}
We then define the Markov kernel $\mathsf{P}_i$ via
\begin{equation*}
    \mathsf{P}_i(\bm{x},A) = \frac{\int_{A_{i,\bm{x}}}\frac{d\pi(x_1,\ldots,x_{i-1},y,x_{i+1},\ldots,x_n)}{d\mu}d\mu_i(y)}{\int_{\mathcal{X}_i}\frac{d\pi(x_1,\ldots,x_{i-1},y,x_{i+1},\ldots,x_n)}{d\mu}d\mu_i(y)}.
\end{equation*}
In words, $\mathsf{P}_i$ rerandomizes the value of the $i$th site according to the conditional distribution of $\pi$ that fixes all other coordinates.
Equivalently, $\mathsf{P}_i$ acts on functions via
\begin{equation}
\label{eq:rerand}
    \mathsf{P}_if(\bm{x}) = \mathbb{E}_{X\sim \pi}[f(X)\vert X_{-i} = \bm{x}_{-i}],
\end{equation}
where we write $-S$ to denote the restriction to the coordinates in $[n]\setminus S$.
It is well-known that $\mathsf{P}_i$ is self-adjoint and satisfies $\mathsf{P}_i^2 = \mathsf{P}_i$. The former is particularly easy to see in the case each $\mathcal{X}_i$ is finite, in which case one can take $\mu_i$ to be uniform and so if $\bm{y}=(x_1,\ldots,x_{i-1},y,x_{i+1},\ldots,x_n)$,
\begin{equation*}
    \pi(\bm{x})\mathsf{P}_i(\bm{x},\bm{y})=\frac{\pi(\bm{x})\pi(\bm{y})}{\sum_{y_i\in \mathcal{X}_i} \pi(x_1,\ldots,x_{i-1},y_i,x_{i+1},\ldots,x_n)}=\mathsf{P}_i(\bm{y},\bm{x})\pi(\bm{y}).
\end{equation*} Therefore, $\mathsf{P}_i$ can be viewed as the orthogonal projection onto the subspace of functions in $L^2(\mathcal{X},\pi)$ that do not depend on $x_i$.

We may now define the relevant operators that will be the subject of this work. For the rest of the paper, we assume that each $\mathsf{P}_i$ is given by \Cref{eq:rerand}.
\begin{definition}
    Let $(\mathcal{X},\mathcal{F})=(\mathcal{X}_1\times\ldots\times \mathcal{X}_n,\otimes_{i=1}^n \mathcal{F}_i)$ and let $\pi$ be some distribution satisfying the above relations so that each $\mathsf{P}_i$ is well-defined. Then the Markov kernel of \textbf{Glauber dynamics} is defined as
    \begin{equation*}
        \mathsf{P}_{\mathsf{GD}} = \frac{1}{n}\sum_{i=1}^n \mathsf{P}_i.
    \end{equation*}
    The \textbf{systematic scan with respect to an order $\sigma:[n]\to [n]$} is defined via
    \begin{equation*}
        \mathsf{P}_{\sigma} = \mathsf{P}_{\sigma(n)}\ldots \mathsf{P}_{\sigma(1)}.
    \end{equation*}
\end{definition}
In words, the Markov kernel associated with Glauber dynamics rerandomizes a uniformly random site, while the systematic scan updates them in the order $\sigma(n),\ldots,\sigma(1)$; note that we write them in reversed order so as to make the application to functions more notationally clear. Note that the Glauber dynamics is self-adjoint, while the systematic scan is a product of orthogonal projections that is generally not itself self-adjoint.

\subsection{Mixing Times}

We now recall several definitions and results from the theory of mixing of Markov chains on finite state spaces. We refer to Levin and Peres~\cite{levin2017markov} or Montenegro and Tetali~\cite{tetali} for comprehensive treatments of the topic. 

\begin{definition}
    Let $\mu,\nu$ be two distributions on a discrete state space $\mathcal{X}$. Then the \textbf{total variation distance} between $\mu$ and $\nu$ is defined as
    \begin{equation*}
        d_{\mathsf{TV}}(\mu,\nu)=\frac{1}{2}\sum_{x\in \mathcal{X}} \vert \mu(x)-\nu(x)\vert.
    \end{equation*}
\end{definition}

\begin{definition}
    For an irreducible and aperiodic Markov chain $\mathsf{P}$ on a finite state space $\mathcal{X}$ with stationary distribution $\pi$, the \textbf{distance to stationarity after $t$} steps is
    \begin{equation*}
        d_{\mathsf{P}}(t):=\max_{x\in \mathcal{X}} d_{\mathsf{TV}}(\mathsf{P}^t(x,\cdot),\pi).
    \end{equation*}

    For any $\varepsilon\in [0,1]$, the \textbf{$\varepsilon$-mixing time} of $\mathsf{P}$ is
    \begin{equation*}
        t_{\mathsf{mix}}^{\mathsf{P}}(\varepsilon)=\min\{t: d_{\mathsf{P}}(t)\leq \varepsilon\}.
    \end{equation*}
\end{definition}

\begin{definition}
    Let $\mathcal{X}$ be a state space and let $\mathsf{P}$ be the Markov kernel of an irreducible Markov chain on $\mathcal{X}$ with stationary measure $\pi$. Then the \textbf{$\pi$-operator norm} of $\mathsf{P}$ is defined as 
    \begin{equation*}
        \|\mathsf{P}\|_{\pi,\mathsf{op}} = \max_{f:\mathbb{E}_{\pi}[f]=0} \frac{\|\mathsf{P}f\|_{\mathsf{\pi}}}{\|f\|_{\pi}}.
    \end{equation*}
    Then the \textbf{spectral gap} of $\mathsf{P}$ is defined by
    \begin{equation*}
        \gamma_{\mathsf{P}} = 1-\|\mathsf{P}\|_{\pi,\mathsf{op}}\geq 0.
    \end{equation*}
We also define the related quantity
    \begin{equation*}
\widetilde{\gamma_{\mathsf{P}}}=\sigma_2(I-\mathsf{P})=\min_{f:\mathbb{E}_{\pi}[f]=0} \frac{\|(I-\mathsf{P})f\|_{\pi}}{\|f\|_{\pi}},
    \end{equation*}
    which is the second smallest singular value of the Laplacian of $\mathsf{P}$ with respect to $\pi$.
\end{definition}

We now have the following results bounding the mixing time of Markov chains:

\begin{theorem}[\cite{levin2017markov}, Theorem 12.4 and 12.5, \cite{tetali}, Proposition 1.14]
\label{thm:markov_mixing}
    Let $\mathcal{X}$ be a finite state space and let $\mathsf{P}$ be the transition matrix of an irreducible Markov chain on $\mathcal{X}$ with stationary measure $\pi$. Then:
    \begin{enumerate}
        \item For any irreducible $\mathsf{P}$,
        \begin{equation*}
             t_{\mathsf{mix}}^{\mathsf{P}}(\varepsilon)\leq \frac{1}{\gamma_{\mathsf{P}}}\log\left(\frac{1}{\varepsilon\pi_{\min}}\right),
        \end{equation*}
        where $\pi_{\min}=\min_{\bm{x}\in \mathcal{X}:\pi(\bm{x})\neq 0} \pi(\bm{x})$.
        \item If $\mathsf{P}$ is reversible with respect to $\pi$, then
        \begin{equation*}
            \left(\frac{1}{\gamma_{\mathsf{P}}}-1\right)\log\left(\frac{1}{2\varepsilon}\right)\leq t_{\mathsf{mix}}^{\mathsf{P}}(\varepsilon).
        \end{equation*}
    \end{enumerate}
\end{theorem}

For reversible Markov chains, the mixing times and operator norm are closely related to the eigenvalues, which all exist and are real by virtue of the spectral theorem. One can provide an analogous lower bound on the mixing time for non-reversible chains in terms of the spectrum, but this is analytically quite difficult since they lie in the unit complex disc. However, we have the recent lower bound of Chatterjee~\cite{chatterjee2023spectral} that will prove useful in providing a full comparison for our setting:

\begin{theorem}[\cite{chatterjee2023spectral}, Theorem 1.4]
\label{thm:nr_lb}
    There is a constant $c>0$ such that the following holds. Let $\mathcal{X}$ be a finite state space and let $\mathsf{P}$ be the transition matrix of an irreducible Markov chain on $\mathcal{X}$ with stationary measure $\pi$.\footnote{Note that Chatterjee defines the inner product structure associate with $\pi$ using complex-valued functions with the conjugate symmetric inner product. However, it is easy to see that for real matrices (like the Laplacian), the second smallest singular value does not depend on whether one considers them over real or complex-valued functions on $\mathcal{X}$ since both the real and complex parts are acted upon separately by linearity, both must have mean zero by the variational formula for singular values, and $\|\cdot\|_{\pi}^2$ is additive across real and complex parts in this case.} Then 
    \begin{equation*}
        t_{\mathsf{mix}}^{\mathsf{P}}(1/10)\geq \frac{c}{\widetilde{\gamma_{\mathsf{P}}}} 
    \end{equation*}
\end{theorem}

For non-reversible chains, there is no clear relation between the singular values of the Laplacian and the eigenvalues of the transition matrix other than the obvious lower bound (see below). However, there is a coarse comparison that is possible in the setting that is the focus of this work:
\begin{lemma}
\label{lem:gibbs_comp}
    Let $\mathcal{X}=\mathcal{X}_1\times\ldots\times \mathcal{X}_n$ be a finite product space and let $\mathsf{P}=\mathsf{P}_n\ldots \mathsf{P}_1$ denote the transition matrix of the systematic scan with identity permutation.\footnote{Again, the actual Markov chain rerandomizes in the order $n,n-1,\ldots,1,n,\ldots$ by our convention.} Then
    \begin{equation*}
        \gamma_{\mathsf{P}}\leq \widetilde{\gamma_{\mathsf{P}}}\leq \sqrt{2n\gamma_{P}}.
    \end{equation*}
\end{lemma}
\begin{proof}
    The first inequality holds for any $\mathsf{P}$; indeed, for any $f$ such that $\mathbb{E}_{\pi}[f]=0$ and $\|f\|_{\pi}=1$, the reverse triangle inequality implies that
    \begin{equation*}
\|(I-\mathsf{P})f\|_{\pi}\geq \|f\|_{\pi}-\|Pf\|_{\pi}\geq 1-\|P\|_{\pi,\mathsf{op}}.
    \end{equation*}
    Since this holds for any such $f$, the first inequality follows.

    The second inequality is a special property of Gibbs samplers. Since each $\mathsf{P}_i$ is an orthogonal projection in $L^2(\mathcal{X},\pi)$, we will be able to precisely track the movement of $f$ under each projection. Let $f$ with $\mathbb{E}_{\pi}[f]=0$ and $\|f\|=1$ attain the maximum for $\|\mathsf{P}\|_{\pi,\mathsf{op}}$ so that
    \begin{equation*}
        (1-\gamma_{\mathsf{P}})^2=\|\mathsf{P}f\|_{\pi}^2
    \end{equation*} 
    Define $\gamma_0 = 0$ and for each $1\leq i\leq n$, set $\mathsf{Q}_i=\mathsf{P}_i\ldots \mathsf{P}_1$ and $\gamma_i$ via
    \begin{equation*}
        1-\gamma_i = \|\mathsf{Q}_if\|_{\pi}^2.
    \end{equation*}
    Note that this sequence is nondecreasing since norms are nonincreasing under projections and that $\gamma_n=2\gamma_{\mathsf{P}}-\gamma_{\mathsf{P}}^2\leq 2\gamma_{\mathsf{P}}$. Moreover, the Pythagorean theorem implies
    \begin{equation*}
        \|\mathsf{Q}_if-\mathsf{Q}_{i-1}f\|_{\pi}^2 = \|\mathsf{Q}_{i-1}f\|_{\pi}^2-\|\mathsf{Q}_if\|_{\pi}^2=\gamma_i-\gamma_{i-1}.
    \end{equation*}
    For this $f$, we thus have by the triangle inequality and Cauchy-Schwarz that
    \begin{align*}
        \|(I-\mathsf{P})f\|_{\pi}^2 &= \|f-\mathsf{Q}_nf\|_{\pi}^2\\
        &=\left\|\sum_{i=1}^n (\mathsf{Q}_if - \mathsf{Q}f)\right\|_{\pi}^2\\
        &\leq \left(\sum_{i=1}^n\left\|\mathsf{Q}_if - \mathsf{Q}_{i-1}f\right\|_{\pi}\right)^2\\
        &\leq n\cdot \sum_{i=1}^n\left\|\mathsf{Q}_if - \mathsf{Q}_{i-1}f\right\|_{\pi}^2\\
        &=n\cdot \sum_{i=1}^n(\gamma_i-\gamma_{i-1})\\
        &= n\cdot (\gamma_n-\gamma_0)\\
        &\leq 2n\gamma_{\mathsf{P}}.
    \end{align*}
Taking roots completes the proof by the variational formula of the second smallest singular value of the Laplacian.
\end{proof}

\begin{corollary}
    There are absolute constants $c,C>0$ such that the following holds. Let $\mathcal{X}=\mathcal{X}_1\times\ldots\times \mathcal{X}_n$ be a finite product space and let $\mathsf{P}=\mathsf{P}_n\ldots \mathsf{P}_1$ denote the transition matrix of the systematic scan with identity permutation. Then
    \begin{equation*}
        \frac{c}{\sqrt{n\gamma_{\mathsf{P}}}}\leq t_{\mathsf{mix}}^{\mathsf{P}}(1/10)\leq \frac{C}{\gamma_{\mathsf{P}}} \log\left(\frac{1}{\pi_{\min}}\right)
    \end{equation*}
\end{corollary}
\begin{proof}
    Simply combine \Cref{thm:markov_mixing}  and \Cref{thm:nr_lb} with \Cref{lem:gibbs_comp}.
\end{proof}

\section{Spectral Comparisons of Scans}
\label{sec:proof}

We now turn to proving our main results via abstract linear-algebraic bounds on operator norms. Throughout this section, for a bounded linear operator $\mathsf{A}:\mathcal{H}\to \mathcal{H}$ on a Hilbert space $\mathcal{H}$, we will write
\begin{equation*}
    \|\mathsf{A}\|_{\mathsf{op}} =\sup_{f\in \mathcal{H}} \frac{\|\mathsf{A}f\|}{\|f\|},
\end{equation*}
where $\|\cdot\|$ is the norm on $\mathcal{H}$ induced by the inner product.

\subsection{From Glauber to All Scans}
In this section, we prove bounds on the operator norm of systematic or pseudorandom scans in terms the spectral gap of Glauber dynamics. By the results of \Cref{sec:preliminaries}, this will yield our optimal mixing time comparison result. We will prove the following linear-algebraic bound from which we will derive all these results:

\begin{theorem}
\label{thm:gen_gap}
        Let $\mathcal{H}$ be a Hilbert space and suppose $\mathsf{P}_1,\ldots, \mathsf{P}_n$ are orthogonal projections with respect to the inner product of $\mathcal{H}$ satisfying
        \begin{equation}
    \left\|\frac{1}{n}\sum_{i=1}^n \mathsf{P}_i\right\|_{\mathsf{op}}\leq 1-\delta.
\end{equation} 
Let $\mathsf{P}_{i_L}\ldots \mathsf{P}_{i_1}$ be any product of the orthogonal projection operators satisfying the following condition. For each $j\in [n]$, suppose there exists a minimal index $k_j\leq L$ satisfying $i_{k_j}=j$. We further assume these indices satisfy\footnote{In words, the indices $i_1,\ldots,i_L$ contain $[n]$ in ascending order of first appearance and are a minimal subsequence with this property.}:
\begin{equation}
\label{eq:order}
        1=k_1\leq k_2\leq\ldots\leq k_n=L.
\end{equation}
Then, it holds that
\begin{equation*}
    \|\mathsf{P}_{i_L}\ldots \mathsf{P}_{i_1}\|_{\mathsf{op}}^2\leq 1-\frac{n\delta}{8\sum_{j=1}^n k_j}.
\end{equation*}
\end{theorem}

Note that \Cref{thm:gen_gap} holds under any other (non-identity) ordering in \Cref{eq:order} of first appearances of the set of indices in $\mathsf{P}_{i_L}\ldots \mathsf{P}_{i_1}$ by symmetry; we choose to write it for the identity ordering simply to avoid additional notation. The following is immediate:

\begin{corollary}
\label{cor:spectral_scan}
    Under the assumptions of \Cref{thm:gen_gap}, it holds that 
    \begin{equation*}
        \|\mathsf{P}_n\ldots \mathsf{P}_1\|_{\mathsf{op}}\leq 1-\frac{\delta}{8(n+1)}.
    \end{equation*}
\end{corollary}
\begin{proof}
    For the systematic scan, it is clear that the assumptions of \Cref{thm:gen_gap} are satisfied since we have $k_j=j$ for each $j\leq n$. Therefore, \Cref{thm:gen_gap} implies that
    \begin{equation*}
        \|\mathsf{P}_n\ldots \mathsf{P}_1\|_{\mathsf{op}}^2\leq 1-\frac{n\delta}{8\sum_{j=1}^n j}=1-\frac{\delta}{4(n+1)}.
    \end{equation*}
    The desired conclusion then follows from Bernoulli's inequality.
\end{proof}

We can now easily conclude our first main result, which we restate for convenience:
\begin{corollary}[\Cref{thm:mixing}, restated]
\label{cor:mixing}
    There exists an absolute constant $C>0$ such that the following holds for any $0<\varepsilon\leq 1$. Let $\mathcal{X}=\mathcal{X}_1\times \ldots\times \mathcal{X}_n$ be a product of finite state spaces and let $\pi$ be any distribution on $\mathcal{X}$. Then
    \begin{equation*}
        \max_{\sigma}t^{\mathsf{SS}(\sigma)}_{\mathrm{mix}}(\varepsilon)\leq Cnt^{\mathsf{GD}}_{\mathrm{mix}}(1/4)\log\left(\frac{1}{\varepsilon\pi_{\mathrm{min}}}\right).
    \end{equation*}
\end{corollary}
\begin{proof}
Let $\gamma_{\mathsf{GD}}>0$ denote the spectral gap of the Glauber dynamics and $\gamma_{\mathsf{SS}(\sigma)}$ the spectral gap of the systematic scan with permutation $\sigma$. \Cref{cor:spectral_scan} instantiated with the Gibbs transition matrices implies 
\begin{equation*}
    \gamma_{\mathsf{SS}(\sigma)}\geq \frac{\gamma_{\mathsf{GD}}}{8(n+1)}.
\end{equation*}
Applying \Cref{thm:markov_mixing} twice (recalling that the Glauber dynamics are reversible) implies that
\begin{equation*}
    t^{\mathsf{SS}(\sigma)}_{\mathrm{mix}}\leq \frac{1}{\gamma_{\mathsf{SS}(\sigma)}}\cdot \log\left(\frac{1}{\varepsilon\pi_{\mathrm{min}}}\right)\leq\frac{8(n+1)}{\gamma_{\mathsf{GD}}}\cdot \log\left(\frac{1}{\varepsilon\pi_{\mathrm{min}}}\right)\leq Cnt^{\mathsf{GD}}_{\mathrm{mix}}(1/4)\log\left(\frac{1}{\varepsilon\pi_{\mathrm{min}}}\right).\qedhere
\end{equation*}
\end{proof}

We also briefly show that the previous arguments also hold for any sufficiently pseudorandom sequences of projections.
\begin{theorem}
\label{thm:pseudo}
    Under the assumptions of \Cref{cor:spectral_scan}, the following holds. Let $i_{\ell}\sim [n]$ be uniform, independent random indices for each $\ell\geq 1$, and let $L$ denote the first index where $\{i_1,\ldots,i_L\}=[n]$. Then there is an event $\mathcal{E}$ such that: (i) $\Pr(\mathcal{E})=1-o(1)$, (ii) on $\mathcal{E}$, it holds that $L\leq 2n\log n$ and 
    \begin{equation*}
        \|\mathsf{P}_{i_L}\ldots \mathsf{P}_{i_1}\|_{\mathsf{op}}\leq 1-\frac{\delta}{32n},
    \end{equation*}
    and (iii) the event $\mathcal{E}$ can be certified to hold in linear time given the indices $i_1,\ldots,i_L$ as input.
\end{theorem}

\begin{proof}
    By changing notation slightly, for each $j\in [n]$, define $k_j$ recursively by
    \begin{gather*}
        k_1 = i_1,\\
        k_{j} = \min\{\ell:i_{\ell}\not\in \{i_1,\ldots,i_{\ell-1}\}\}.
    \end{gather*}
    These indices denote the first times a new index appears in the sequence and $k_n=L$ by definition.

    Let $\mathcal{E}$ denote the event that $k_n\leq 2n \log n$ and that
    \begin{equation*}
        \sum_{j=1}^n k_j\leq 2 n^2.
    \end{equation*}
    Note that by symmetry of indices, \Cref{thm:gen_gap} clearly still implies that on this event,
    \begin{equation*}
         \|\mathsf{P}_{i_L}\ldots \mathsf{P}_{i_1}\|_{\mathsf{op}}^2\leq 1-\frac{n\delta}{16n^2}\leq 1-\frac{\delta}{16n}.
    \end{equation*}
    Applying Bernoulli's inequality thus shows that the desired operator norm bound holds. As the event $\mathcal{E}$ is clearly easy to certify given the sequence of updated indices in linear time, it suffices to prove $\Pr(\mathcal{E})=1-o(1)$.

    We do this by providing a crude second moment bound; certainly tighter analyses are possible at the cost of further computation. We first rewrite
    \begin{align*}
        \sum_{j=1}^{n} k_j &= \sum_{j=1}^{n}\sum_{\ell=1}^j (k_{\ell}-k_{\ell-1})\\
        &=\sum_{\ell=1}^n (n-\ell+1)(k_{\ell}-k_{\ell-1}),
    \end{align*}
    where we define $k_0=0$. It is well-known and elementary from a standard coupon collector argument that for $\ell=1,\ldots, n$, $k_{\ell}-k_{\ell-1}\sim \mathsf{Geom}(p_{\ell})$, where $p_{\ell}=\frac{n-\ell+1}{n}$. We thus have
    \begin{gather*}
        \mathbb{E}[k_{\ell}-k_{\ell-1}]=\frac{1}{p_{\ell}},\\
        \mathrm{Var}(k_{\ell}-k_{\ell-1})=\frac{1-p_{\ell}}{p_{\ell}^2}.
    \end{gather*}
    The first identity implies
    \begin{align*}
        \mathbb{E}\left[\sum_{\ell=j}^nk_{j}\right]&=\sum_{j=1}^n\sum_{\ell=1}^j\mathbb{E}[k_{\ell}-k_{\ell-1}]\\
        &=\sum_{j=1}^n\sum_{\ell=1}^j \frac{1}{p_{\ell}}\\
        &=\sum_{\ell=1}^n \frac{n-\ell+1}{p_{\ell}}\\
        &=n^2.
    \end{align*}
    
    These increments are also independent, and so
    \begin{align*}
        \mathrm{Var}\left(\sum_{j=1}^{n} k_j\right)&=\mathrm{Var}\left(\sum_{\ell=1}^n (n-\ell+1)(k_{\ell}-k_{\ell-1})\right)\\
        &=\sum_{\ell=1}^n (n-\ell+1)^2\mathrm{Var}(k_{\ell}-k_{\ell-1})\\
        &=n^2\sum_{\ell=1}^n (1-p_{\ell})\\
        &\leq n^3.
    \end{align*}

    It follows from Chebyshev's inequality that
    \begin{align*}
        \Pr\left(\sum_{j=1}^n k_j\geq 2n^2\right)\leq \Pr\left(\left\vert\sum_{j=1}^n k_j-\mathbb{E}\left[\sum_{j=1}^n k_j\right]\right\vert\geq n^2\right)\leq \frac{1}{n}.
    \end{align*}
    It is also classical that $\Pr(k_n\geq 2n\log n)\leq \frac{1}{n}$ by a standard coupon collector analysis~\cite{levin2017markov}. By a union bound, it thus follows that $\Pr(\mathcal{E})\geq 1-2/n$, as desired.    
\end{proof}

With these applications in order, we may return to the proof of \Cref{thm:gen_gap}:

\begin{proof}[Proof of \Cref{thm:gen_gap}]
We use similar notation as in \Cref{lem:gibbs_comp}. Fix any $f\in \mathcal{H}$ satisfying $\|f\|=1$ and define the sequence $0:=\gamma_0\leq \gamma_1\leq\ldots\leq \gamma_L$ via
\begin{equation*}
    \|\mathsf{P}_{i_{\ell}}\ldots \mathsf{P}_{i_1}f\|^2=1-\gamma_{\ell}.
\end{equation*}
For convenience of notation, further define the operators $\mathsf{Q}_{\ell}$ via 
\begin{equation*}
    \mathsf{Q}_{\ell} = \mathsf{P}_{i_{\ell}}\ldots \mathsf{P}_{i_1}.
\end{equation*}

For each $\ell\leq L-1$, the orthogonal decomposition
    \begin{equation*}
        \mathsf{Q}_{{\ell}}f = \mathsf{P}_{i_{\ell+1}}( \mathsf{Q}_{{\ell}}f)+(\mathsf{Q}_{{\ell}}f-P_{i_{\ell+1}}\mathsf{Q}_{{\ell}} f)= \mathsf{Q}_{{\ell+1}}f+(\mathsf{Q}_{{\ell}}f-\mathsf{Q}_{{\ell+1}} f)
    \end{equation*}
    certifies that
    \begin{equation}
    \label{eq:telescope}
        \|\mathsf{Q}_{{\ell}}f-\mathsf{Q}_{{\ell+1}} f\|^2=\|\mathsf{Q}_{{\ell}}f\|^2-\|\mathsf{Q}_{{\ell+1}}f\|^2=\gamma_{\ell+1}-\gamma_{\ell}.
    \end{equation}

    We now turn to bounding $\|\mathsf{P}_j f-f\|$. Since $i_{k_j}=j$ by definition,
    \begin{align*}
        \mathsf{P}_j f-f &= \mathsf{P}_{j}(f-\mathsf{P}_{i_{k_{j}-1}}\ldots \mathsf{P}_{i_1} f) + (\mathsf{P}_{i_{k_j}}\ldots \mathsf{P}_{i_1}f -f)\\
        &= \mathsf{P}_{j}(f-\mathsf{Q}_{{k_{j}-1}} f) + (\mathsf{Q}_{{k_j}}f -f)\\
        &=\mathsf{P}_{j}\left(\sum_{\ell=1}^{k_j-1} (\mathsf{Q}_{{\ell-1}} f-\mathsf{Q}_{{\ell}} f)\right)+\sum_{\ell=1}^{k_j} (\mathsf{Q}_{{\ell}} f -  \mathsf{Q}_{{\ell-1}} f).
    \end{align*}

    Taking norms, applying the triangle inequality and using the fact $\mathsf{P}_j$ is a projection, we obtain
\begin{align*}
    \| \mathsf{P}_j f-f\|&\leq \left\| \mathsf{P}_{j}\left(\sum_{\ell=1}^{k_j-1} (\mathsf{Q}_{{\ell-1}} f-\mathsf{Q}_{{\ell}} f)\right)\right\|+\sum_{\ell=1}^{k_j} \left\|\mathsf{Q}_{{\ell}} f -  \mathsf{Q}_{{\ell-1}} f\right\|\\
    &\leq  \sum_{\ell=1}^{k_j-1} \left\|\mathsf{Q}_{{\ell-1}} f-\mathsf{Q}_{{\ell}} f\right\|+\sum_{\ell=1}^{k_j} \left\|\mathsf{Q}_{{\ell}} f -  \mathsf{Q}_{{\ell-1}} f\right\|\\
    &\leq 2 \sum_{\ell=1}^{k_j} \left\|\mathsf{Q}_{{\ell}} f -  \mathsf{Q}_{{\ell-1}} f\right\|.
\end{align*}
Taking squares, applying Cauchy-Schwarz, and telescoping via \Cref{eq:telescope}, we find that
\begin{align*}
    \| \mathsf{P}_j f-f\|^2&\leq \left(2 \sum_{\ell=1}^{k_j} \left\|\mathsf{Q}_{{\ell}} f -  \mathsf{Q}_{{\ell-1}} f\right\|\right)^2\\
    &\leq 4\left(\sqrt{k_j}\sqrt{\sum_{\ell=1}^{k_j} \left\|\mathsf{Q}_{{\ell}} f -  \mathsf{Q}_{{\ell-1}} f\right\|^2}\right)^2\\
    &=4k_j\sum_{\ell=1}^{k_j} \left\|\mathsf{Q}_{{\ell}} f -  \mathsf{Q}_{{\ell-1}} f\right\|^2\\
    &=4k_j\sum_{\ell=1}^{k_j} (\gamma_{\ell}-\gamma_{\ell-1})\\
    &=4k_j\gamma_{k_j}.
\end{align*}

Finally, we observe by properties of orthogonal projection and the assumption $\|f\|=1$ that
\begin{align*}
    \left\| \frac{1}{n}\sum_{j=1}^n \mathsf{P}_jf\right\|^2&=\left\|f+ \frac{1}{n}\sum_{j=1}^n (\mathsf{P}_jf-f)\right\|^2\\
    &=\|f\|^2+2\left\langle f,\frac{1}{n}\sum_{j=1}^n (\mathsf{P}_jf-f)\right\rangle+\left\|\frac{1}{n}\sum_{j=1}^n (\mathsf{P}_jf-f)\right\|^2\\
    &\geq \|f\|^2-\frac{2}{n}\sum_{j=1}^n \vert \langle f,\mathsf{P}_jf-f\rangle\vert\\
    &=1-\frac{2}{n}\sum_{j=1}^n \|\mathsf{P}_jf-f\|^2\\
    &\geq 1-\frac{8}{n}\sum_{j=1}^n k_j\gamma_{k_j}\\
    &\geq 1-\frac{8\gamma_{L}}{n}\cdot \sum_{j=1}^n k_j.
\end{align*}
To see the last equality, note that
\begin{equation*}
    \langle f,\mathsf{P}_jf-f\rangle = \langle f-\mathsf{P}_jf,\mathsf{P}_jf-f\rangle+ \langle\mathsf{P}_jf,\mathsf{P}_jf-f\rangle = -\|\mathsf{P}_jf-f\|^2,
\end{equation*}
since any orthogonal projection is orthogonal to the remainder by the Pythagorean property.
Finally, by assumption the left hand side is at most $(1-\delta)^2\leq 1-\delta$ so rearranging yields
\begin{equation*}
    \gamma_{L} \geq \frac{n\delta}{8\sum_{j=1}^n k_j}.
\end{equation*}
Since this holds for any $f\in \mathcal{H}$ of unit norm, we have proven the claim.
\end{proof}

\subsection{From Scans to Glauber}

We now show how a variation of the argument of \Cref{thm:gen_gap} can be used to show a converse: if all, or even any, systematic scans have a spectral gap, then the Glauber dynamics has a spectral gap that is polynomially related. This provides a simple resolution to Open Problem 3.7 of Chlebicka,  Łatuszy\'nski, and Miasodejow~\cite{chlebicka}. However, unlike \Cref{thm:mixing}, we expect that our quantitative dependence is quantitatively suboptimal.

\begin{theorem}
\label{thm:converse}
    Let $\mathcal{H}$ be a Hilbert space and suppose $\mathsf{P}_1,\ldots,\mathsf{P}_n$ are orthogonal projections with respect to the inner product of $\mathcal{H}$. There is an absolute constant $c_{\ref{thm:converse}}>0$ such that the following holds for any $\delta>0$:
    \begin{enumerate}
        \item \label{item:all} Suppose that for every permutation $\sigma:[n]\to [n]$, it holds that
        \begin{equation*}
            \left\|\mathsf{P}_{\sigma(n)}\ldots \mathsf{P}_{\sigma(1)}\right\|_{\mathsf{op}}\leq 1-\delta.
        \end{equation*}
        Then 
        \begin{equation*}
            \left\|\frac{1}{n}\sum_{i=1}^n \mathsf{P}_i\right\|\leq 1-c_{\ref{thm:converse}}\left(\frac{\delta}{n\log(n/\delta)}\right)^2.
        \end{equation*}
        \item \label{item:one} Suppose that there exists a  permutation $\sigma:[n]\to [n]$ such that
        \begin{equation*}
            \left\|\mathsf{P}_{\sigma(n)}\ldots \mathsf{P}_{\sigma(1)}\right\|_{\mathsf{op}}\leq 1-\delta.
        \end{equation*}
        Then 
        \begin{equation*}
            \left\|\frac{1}{n}\sum_{i=1}^n \mathsf{P}_i\right\|\leq 1-c_{\ref{thm:converse}}\left(\frac{\delta}{n^2\log(n/\delta)}\right)^2.
        \end{equation*}
    \end{enumerate}
\end{theorem}

We will derive these results as a consequence of a general linear-algebraic result relating products of projections to that of an embedded permutation of the projections. Our assumption is that this embedded permutation has a spectral gap, and therefore there is an explicit spectral gap of the larger product that depends only on the original gap and the number of products in the sequence. The argument is conceptually straightforward: if the larger sequence of projections does not behave similarly to the single embedded permutation, this must be because the extra applications of projections noticeably moved the vector. But since any projection that moves a vector must significantly decrease the norm, this movement can also be charged to the spectral gap, and therefore one can obtain a polynomial relationship between the spectral gaps.

\begin{theorem}
\label{thm:gen_gap_2}
        Let $\mathcal{H}$ be a Hilbert space and suppose $\mathsf{P}_1,\ldots, \mathsf{P}_n$ are orthogonal projections with respect to the inner product of $\mathcal{H}$. Suppose that
        \begin{equation}
    \left\|\mathsf{P}_n\ldots \mathsf{P}_1\right\|_{\mathsf{op}}\leq 1-\delta.
\end{equation} 
Let $\mathsf{P}_{i_L}\ldots \mathsf{P}_{i_1}$ be any product of these orthogonal projection operators satisfying the following condition. For each $j\in [n]$, suppose there exists a minimal index $k_j\leq L$ satisfying $i_{k_j}=j$. We further assume these indices satisfy:
\begin{equation}
\label{eq:order_2}
        1=k_1\leq k_2\leq\ldots\leq k_n=L.
\end{equation}
Then, it holds that
\begin{equation*}
    \|\mathsf{P}_{i_L}\ldots \mathsf{P}_{i_1}\|_{\mathsf{op}}\leq 1-\frac{\delta^2}{8(L-n+1)}.
\end{equation*}
\end{theorem}

The square is likely unnecessary and is an artifact of the analysis; roughly speaking, the reason it arises is that we work with regular distances using the triangle inequality rather than squared distances as in \Cref{thm:gen_gap} due to less favorable orthogonality properties in the analysis. This issue already manifests even in the simple case $L=n$, where the bound should trivially be $1-\delta$.
\begin{proof}
    We will provide two lower bounds on the spectral gap and show that at least one gives the desired value. As in \Cref{thm:gen_gap}, fix any $f\in \mathcal{H}$ satisfying $\|f\|=1$, set $\mathsf{Q}_{\ell} = P_{i_{\ell}}\cdots P_{i_1}$ and define the sequence $0=\lambda_0\leq \ldots\leq \lambda_L$ via
\begin{equation*}
    \|\mathsf{Q}_{{\ell}}f\|^2 = 1-\lambda_{\ell}.
\end{equation*}
Our goal is to lower bound $\lambda_L$. An identical argument to \Cref{thm:gen_gap} by writing out the orthogonal decompositions implies that for any $\ell$,
\begin{equation}
\label{eq:telescope_2}
    \|\mathsf{Q}_{{\ell+1}}f - \mathsf{Q}_{{\ell}}f\|^2 = \lambda_{\ell+1} - \lambda_{\ell}.
\end{equation}
On the one hand, for any subset $S$ of $[L]$, we trivially have by nonnegativity of the summands that
\begin{equation}
\label{eq:bound_1}
    \lambda_L = \sum_{\ell=1}^L (\lambda_{\ell}-\lambda_{\ell-1})\geq \sum_{\ell\in S} (\lambda_{\ell}-\lambda_{\ell-1}).
\end{equation}

On the other hand, since $i_L=n$ by assumption, we have
\begin{align*}
    \mathsf{Q}_L f &=\mathsf{P}_{i_{L}}\ldots \mathsf{P}_{i_1}f\\
    &= \mathsf{P}_n \mathsf{Q}_{{L-1}}f\\
    &=\mathsf{P}_n\mathsf{P}_{n-1}\ldots \mathsf{P}_1f + \mathsf{P}_n(\mathsf{Q}_{{L-1}}f-\mathsf{P}_{n-1}\ldots \mathsf{P}_1f).
\end{align*}
Taking norms, the triangle inequality, and the fact projections cannot increase norm imply that
\begin{align}
    \|\mathsf{Q}_{{L}}f\|&\leq \|\mathsf{P}_n\mathsf{P}_{n-1}\ldots \mathsf{P}_1f\| + \|\mathsf{P}_n(\mathsf{Q}_{{L-1}}f-\mathsf{P}_{n-1}\ldots \mathsf{P}_1f)\| \nonumber\\
    &\leq 1-\delta + \|\mathsf{Q}_{{L-1}}f-\mathsf{P}_{n-1}\ldots \mathsf{P}_1f\| \label{eq:initial}.
\end{align}
We now treat the last term similarly as before: write
\begin{equation*}
    \mathsf{Q}_{{L-1}}f = \mathsf{Q}_{{k_{n-1}}}f+\sum_{\ell=k_{n-1}}^{L-2} (\mathsf{Q}_{{\ell+1}}f-\mathsf{Q}_{{\ell}}f) ,
\end{equation*}
so a similar set of manipulations with the triangle inequality, Cauchy-Schwarz, and the fact projections can only decrease distance (noting $i_{k_{n-1}}=n-1$ by definition), and \Cref{eq:telescope_2} yields
\begin{align*}
    \|\mathsf{Q}_{{L-1}}f-\mathsf{P}_{n-1}\ldots \mathsf{P}_1f\|&=\left\| \mathsf{Q}_{{k_{n-1}}}f-\mathsf{P}_{n-1}\ldots \mathsf{P}_1f+\sum_{\ell=k_{n-1}}^{L-2} (\mathsf{Q}_{{\ell+1}}f-\mathsf{Q}_{{\ell}}f)\right\|\\
    &\leq \| \mathsf{Q}_{{k_{n-1}}}f-\mathsf{P}_{n-1}\ldots \mathsf{P}_1f\|+\sum_{\ell=k_{n-1}}^{L-2} \|\mathsf{Q}_{{\ell+1}}f-\mathsf{Q}_{{\ell}}f\|\\
    &=\| \mathsf{P}_{n-1}\mathsf{Q}_{{k_{n-1}-1}}f-\mathsf{P}_{n-1}\ldots \mathsf{P}_1f\|+\sqrt{k_{n}-k_{n-1}-1} \sqrt{\sum_{\ell=k_{n-1}}^{L-2} \|\mathsf{Q}_{{\ell+1}}f-\mathsf{Q}_{{\ell}}f\|^2}\\
    &\leq \|\mathsf{Q}_{{k_{n-1}-1}}f-\mathsf{P}_{n-2}\ldots \mathsf{P}_1f\|+\sqrt{k_n-k_{n-1}-1}\sqrt{\sum_{\ell=k_{n-1}}^{L-2} (\lambda_{\ell+1}-\lambda_{\ell})}.
\end{align*}

Proceeding in the same way recursively to deal with the first term (and using $i_1=1$ at the end of the recursion), we then obtain by Cauchy-Schwarz that
\begin{align*}
    \|\mathsf{Q}_{{L-1}}f-\mathsf{P}_{n-1}\ldots \mathsf{P}_1f\|&\leq \sum_{j=1}^{n-1} \sqrt{k_{j+1}-k_{j}-1}\sqrt{\sum_{\ell=k_{j}}^{k_{j+1}-2} (\lambda_{\ell+1}-\lambda_{\ell})} \nonumber\\
    &\leq \sqrt{\sum_{j=1}^{n-1} (k_{j+1}-k_j-1)}\sqrt{\sum_{\ell\in S}(\lambda_{\ell+1}-\lambda_{\ell})}\nonumber\\
    &=\sqrt{L-n}\sqrt{\sum_{\ell\in S}(\lambda_{\ell+1}-\lambda_{\ell})},
\end{align*}
where $S$ is a suitable subset of $[L]$; note that we use the fact that the inner summands for the increments are disjoint across different $k$ by construction. Putting this back with \Cref{eq:initial}, we have
\begin{equation}
\label{eq:bound_2}
    \|\mathsf{P}_{i_{L}}\ldots \mathsf{P}_{i_1}f\|\leq 1-\delta+ \sqrt{L-n}\sqrt{\sum_{\ell\in S}(\lambda_{\ell+1}-\lambda_{\ell})}
\end{equation}

For this subset $S$, let 
\begin{equation*}
    \kappa = \sum_{\ell\in S}(\lambda_{\ell+1}-\lambda_{\ell}).
\end{equation*}
By \Cref{eq:bound_1} and \Cref{eq:bound_2}, we observe that
\begin{equation*}
    \lambda_L \geq \max\left\{\kappa,\delta- \sqrt{L-n}\kappa^{1/2} \right\}.
\end{equation*}
It is not difficult to see that the right hand side  has value at least $\delta^2/4(L-n+1)$ by a simple case analysis (note  though that the bound improves to $\delta$ when $L=n$, as it should). Therefore, we conclude that 
\begin{equation*}
    \lambda_L\geq \frac{\delta^2}{4(L-n+1)}
\end{equation*}
as claimed, which yields the desired operator norm bound (without the square) by Bernoulli's inequality.
\end{proof}

We now show how to derive a quantitative spectral gap for the Glauber dynamics.

\begin{proof}[Proof of \Cref{thm:converse}]
    The idea is to take powers and argue that most of the terms contain one of the permutations as an embedded subsequence that is assumed to have the requisite gap. The spectral gap from the previous line will decay linearly in the power, but the probability of failing to contain the subsequence will decay like a large polynomial, so we can set parameters appropriately. We may assume $n\geq 2$ throughout.

    More formally, since by self-adjointness, we have
    \begin{equation*}
        \left\|\frac{1}{n}\sum_{i=1}^n \mathsf{P}_i\right\|^L =\left\|\left(\frac{1}{n}\sum_{i=1}^n \mathsf{P}_i\right)^L\right\| ,
    \end{equation*}
    it will suffice to upper bound the latter. For any $L\geq n$, we first write
    \begin{equation*}
        \left(\frac{1}{n}\sum_{i=1}^n \mathsf{P}_i \right)^L = \frac{1}{n^L}\sum_{i_1,\ldots,i_L=1}^n \mathsf{P}_{i_L}\ldots \mathsf{P}_{i_1}.
    \end{equation*}

    We now choose parameters for each case:
    \begin{itemize}
        \item First suppose the condition of \Cref{item:all} so that all scans have an operator norm gap. For any $L\geq n$, define 
        \begin{equation*}
            \mathcal{A} =\left\{(i_L,\ldots,i_1)\in [n]^L: [n] = \{i_L,\ldots,i_1\}\right\}.
        \end{equation*}
        In words, this is the set of all sequences of indices containing all indices in $[n]$. By standard coupon collector bounds, if we choose $L=5n \log(200n/\delta)$, then
        \begin{equation*}
            \Pr((i_L,\ldots,i_1)\not\in \mathcal{A})\leq \frac{\delta^{3}}{200n^{3}}.
        \end{equation*}
        On the other hand, for any such tuple $(i_L,\ldots,i_1)\in \mathcal{A}$, we claim
        \Cref{thm:gen_gap_2} implies that
        \begin{equation*}
            \|\mathsf{P}_{i_L}\ldots \mathsf{P}_{i_1}\|_{\mathsf{op}}\leq 1-\frac{\delta^2}{8L}=1-\frac{\delta^2}{40n\log(200n/\delta)}.
        \end{equation*}
        Indeed, this tuple contains all elements in some order by definition of $\mathcal{A}$, and we may assume $i_L=\sigma(n)$ is the last element found and $i_1=\sigma(1)$ is the first element of this permutation.  The reason is that the operator norm can only decrease by removing a prefix (on the left) of operators since projections cannot increase the norm and by removing any suffix of projectors (on the right) since these projectors map the unit ball of $\mathcal{H}$ to itself. But then by symmetry, we can directly apply \Cref{thm:gen_gap_2} to conclude the bound. It follows that
        \begin{align*}
            \left\|\left(\frac{1}{n}\sum_{i=1}^n \mathsf{P}_i\right)^L\right\|&\leq \frac{1}{n^L}\sum_{(i_L,\ldots,i_1)\in \mathcal{A}} \|\mathsf{P}_{i_L}\ldots \mathsf{P}_{i_1}\|_{\mathsf{op}}+\frac{1}{n^L}\sum_{(i_L,\ldots,i_1)\in \mathcal{A}^c} \|\mathsf{P}_{i_L}\ldots \mathsf{P}_{i_1}\|_{\mathsf{op}}\\
            &\leq 1-\frac{\delta^2}{40n\log(200n/\delta)}+\Pr((i_L,\ldots,i_1)\not\in \mathcal{A})\\
            &\leq 1-\frac{\delta^2}{40n\log(200n/\delta)}+\frac{\delta^{3}}{200n^3}\\
            &\leq 1-\frac{\delta^2}{100n\log(200n/\delta)}.
        \end{align*}
        Here, we simply apply the triangle inequality and crudely bound the operator norms by $1$ for those elements not in $\mathcal{A}$. Therefore,
        \begin{equation*}
            \left\|\frac{1}{n}\sum_{i=1}^n \mathsf{P}_i\right\|\leq \left(1-\frac{\delta^2}{100n\log(200n/\delta)}\right)^{1/L}\leq 1-\frac{c\delta^2}{n^2\log^2(n/\delta)}
        \end{equation*}
        for some absolute constant $c>0$.

        \item Next, we assume that there exists a permutation $\sigma:[n]\to [n]$ such that the associated product of the projections has operator norm at most $1-\delta$. We use the same approach as last time, now just with $L = Cn^2\log(100Cn/\delta)$ for a suitable constant $C>0$ to be determined later. Let $\mathcal{A}$ denote the set of tuples $(i_L,\ldots,i_1)$ containing $\sigma$ as an ordered subsequence. Note that the probability that $(i_L,\ldots,i_1)\not\in \mathcal{A}$ is precisely the probability that a $\mathsf{Bin}(L,1/n)$ random variable is less than $n$ by the natural coupling that interprets each success as sampling the next element of the permutation $\sigma$. For this choice of $L$, it is easy to see from Hoeffding's inequality that if $C$ is a large enough constant, then we have
        \begin{equation*}
            \Pr((i_L,\ldots,i_1)\not\in \mathcal{A})\leq \delta^{100}/100Cn^{100}.
        \end{equation*}
        The same argument as the previous case asserts that for any $(i_L,\ldots,i_1)\in \mathcal{A}$, we have an operator norm bound of $1-\delta^2/8Cn^2\log(100Cn/\delta)$. Therefore, the same argument and calculation splitting the sum across $\mathcal{A}$ yields
        \begin{align*}
            \left\|\left(\frac{1}{n}\sum_{i=1}^n \mathsf{P}_i\right)^L\right\|&\leq \frac{1}{n^L}\sum_{(i_L,\ldots,i_1)\in \mathcal{A}} \|\mathsf{P}_{i_L}\ldots \mathsf{P}_{i_1}\|_{\mathsf{op}}+\frac{1}{n^L}\sum_{(i_L,\ldots,i_1)\in \mathcal{A}^c} \|\mathsf{P}_{i_L}\ldots \mathsf{P}_{i_1}\|_{\mathsf{op}}\\
            &\leq 1-\frac{\delta^2}{8Cn^2\log(100Cn/\delta)}+\Pr((i_L,\ldots,i_1)\not\in \mathcal{A})\\
            &\leq 1-\frac{\delta^2}{8Cn^2\log(100Cn/\delta)}+\frac{\delta^{100}}{100Cn^{100}}\\
            &\leq 1-\frac{c\delta^2}{n^2\log(n/\delta)}.
        \end{align*}
        for a suitable constant $c>0$. Taking $1/L$ norms again gives the stated bound.\qedhere
    \end{itemize}
\end{proof}

We may now conclude our mixing time comparison:

\begin{corollary}[\Cref{thm:mixing_2}, formal]
    The following holds for any $0<\varepsilon\leq 1$. Let $\mathcal{X}=\mathcal{X}_1\times \ldots\times \mathcal{X}_n$ be a product of finite state spaces and let $\pi$ be any distribution on $\mathcal{X}$. For a permutation $\sigma$, again let $t_{\mathsf{mix}}^{\mathsf{SS}(\sigma)}(\cdot)$ denote the mixing time of the systematic scan with ordering $\sigma$. Then we have the following bounds:
    \begin{gather*}
        t^{\mathsf{GD}}_{\mathsf{mix}}(\varepsilon)\leq \max_{\sigma}\widetilde{O}\left(n^4(t_{\mathsf{mix}}^{\mathsf{SS}(\sigma)}(1/10))^4\right)\cdot \log\left(\frac{1}{\varepsilon\pi_{\min}}\right)\\
        t^{\mathsf{GD}}_{\mathsf{mix}}(\varepsilon)\leq \min_{\sigma}\widetilde{O}\left(n^6(t_{\mathsf{mix}}^{\mathsf{SS}(\sigma)}(1/10))^4\right)\cdot \log\left(\frac{1}{\varepsilon\pi_{\min}}\right).
    \end{gather*}
\end{corollary}
\begin{proof}
    This is an immediate consequence of \Cref{thm:markov_mixing} by using \Cref{thm:converse} to provide a mixing time bound in terms of spectral gaps of appropriate scans, and then applying \Cref{thm:nr_lb} to upper bound in terms of the mixing times.
\end{proof}

We note that we lose polynomial factors since the translation between operator norms and mixing times of operator norms in the non-reversible case is qualititatively weaker than in the reversible case; this is compounded by the polynomial losses in the spectral gaps from the above. It would be very interesting to determine the precise polynomial relationship between these quantities.

\section{Optimality of Results}
\label{sec:optimality}

We now show that none of the assumptions or conclusions of \Cref{cor:spectral_scan} can be significantly improved, which also manifests in tightness of \Cref{thm:mixing}. In particular, the result crucially exploits the orthogonal projection properties and the factor of $n$ loss is necessary. For future work, it would be quite interesting to understand general quantitative relationships for other functional inequalities, like (modified) log-Sobolev inequalities, as well as whether one can prove comparable mixing bounds (per step, up to necessary logarithmic factors in dimension) by leveraging system-specific properties.\\

\noindent\textbf{Necessity of Projections.} First, we observe that it is crucial that $\mathsf{P}_i$ are projections rather than arbitrary positive semi-definite operators. In fact, the following simple example of Recht and R{\'{e}}~\cite{DBLP:journals/jmlr/RechtR12} shows it is possible for the assumption of \Cref{cor:spectral_scan} to hold with $\delta>0$ for general linear operators, yet $\mathsf{P}_n\ldots \mathsf{P}_1$ has operator norm exponential in $n$. 

Given $n$ and $\delta\in [0,1]$, let $\omega_n=\pi/n$ and for $k=1,\ldots,n,$ define
\begin{equation*}
    a_{k,n}=\sqrt{2(1-\delta)} (\cos(k\omega_n),\sin(k\omega_n))^T\in \mathbb{R}^2.
\end{equation*}
Define $A_i=a_{k,n}a_{k,n}^T$; these matrices are obviously positive semi-definite, but are not projections (with respect to the usual inner product structure) unless the $a_{k,n}$ are unit vectors (i.e. $\delta=1/2$). One can check from Recht and R{\'{e}}~\cite{DBLP:journals/jmlr/RechtR12} that these vectors are isotropic up to scaling:
    \begin{equation*}
        \frac{1}{n}\sum_{i=1}^n A_i = (1-\delta)\cdot I,
    \end{equation*}
    while 
    \begin{equation*}
        A_n\ldots A_1=(2(1-\delta)\cos(\omega_n))^{n-1}a_{n,n}a_{1,n}^T.
    \end{equation*}
    Since this matrix is rank-one, it is immediate that 
    \begin{equation*}
    \|A_n\ldots A_1\|_{\mathsf{op}} = (2(1-\delta))^n\cos(\omega_n)^{n-1}.
    \end{equation*}
    Note that
    \begin{equation*}
        \lim_{n\to \infty} \cos(\omega_n)^{n-1}=1,
    \end{equation*}
    so this certifies the exponentially large operator norm of the matrix product for any fixed $0<\delta<1/2$. \\
    
    \noindent\textbf{Optimality of \Cref{cor:spectral_scan}.} In the previous example, note that when $\delta=1/2$, the $A_i$ are indeed orthogonal projection operators, so are within the purview of the general linear-algebraic result. In this case, we have $\cos(\omega_n)^{n-1}=1-\frac{\pi^2}{2n}+O\left(\frac{1}{n^2}\right)$ and so the factor of $n$ loss in \Cref{cor:spectral_scan} is necessary.\\

    \noindent\textbf{Optimality of \Cref{thm:mixing}.}
    Finally, we observe that the loss in mixing times from \Cref{thm:mixing} can be realized in natural spin systems via the following simple example of Roberts and Rosenthal~\cite{roberts2015surprising}, up to the logarithmic factor that is inherent to spectral bounds on mixing times. For completeness, we describe this example and fill in the details of their mixing time analysis for the systematic scan. 
    
    Let $\pi$ be the uniform distribution on the set of spin configurations $\bm{x}\in \{0,1\}^n$ such that $\sum_{i=1}^n x_i\leq 1$. This is the so-called \emph{hardcore model} of independent sets on the complete graph $K_n$ on $n$ nodes with fugacity $\lambda=1$. In particular, the set of configurations with positive (and uniform) probability is $\{\bm{0},\bm{e}_1,\ldots,\bm{e}_n\}$, where $\bm{e}_j$ is the $j$th standard basis vector in $\mathbb{R}^n$. The associated Gibbs sampler is easily seen to act on configurations as follows when re-randomizing a spin $i$ at time $t$: if $X^t\in \{\bm{0},\bm{e}_i\}$, $X^{t+1}$ is uniform in $\{\bm{0},\bm{e}_i\}$, while otherwise $X^{t+1}=X^t$.

    \begin{theorem}[\cite{roberts2015surprising}, Section 5]
        There exists an absolute constant $c>0$ such that for the hardcore model on $K_n$ with fugacity $1$, and for any scan order $\sigma:[n]\to [n]$,
        \begin{equation*}
            t^{\mathsf{SS}(\sigma)}_{\mathrm{mix}}(1/4)\geq cnt^{\mathsf{GD}}_{\mathrm{mix}}(1/4).
        \end{equation*}
        In particular, \Cref{thm:mixing} is tight up to a factor of $\log(1/\pi_{\mathrm{min}})$.
    \end{theorem}

\begin{proof}
    We follow the ideas sketched by Roberts and Rosenthal~\cite{roberts2015surprising}. For $t_{\mathrm{mix}}^{\mathsf{GD}}$, their work shows that the total variation distance between $T$ steps of Glauber dynamics and $\pi$ is bounded by $\Pr(U+V>T)$, where $U\sim \mathsf{Geom}(1/2n)$ and $V\sim \mathsf{Geom}(1/2)$. It follows that $t_{\mathrm{mix}}^{\mathsf{GD}}(1/4)\leq C_{\mathsf{GD}}n$ for a large enough constant $C_{\mathsf{GD}}>0$. Therefore, it suffices to show that, for some small enough constant $c'>0$, the total variation distance between $c'n^2$ full sweeps of the systematic scan dynamics and $\pi$ is greater than $1/4$. We will prove this for $\sigma:[n]\to [n]$ the identity permutation; it is easy to see that the argument holds for any other scan order by symmetry up to an appropriate relabeling of sites.

    For this, we follow the argument of Roberts and Rosenthal~\cite{roberts2015surprising} that describes an alternative way to view the dynamics of the systematic scan. Suppose $X^0=\bm{e}_n$. Define the following stopping times for $s=1,\ldots$:

    \begin{gather*}
        \tau_1 = \min\{t\geq 1: X^t =\bm{0}\}\quad\quad \nu_1 =\min\{t\geq \tau_1: X^t \neq\bm{0}\}\\
        \tau_{s+1}=\min\{t\geq \nu_s: X^t =\bm{0}\}\quad\quad \nu_{s+1}=\min\{t\geq \tau_{s+1}: X^t\neq\bm{0}\}
    \end{gather*}

    The idea is as follows: $\tau_s$ is the $s$th time that the configuration drops down to $\bm{0}$. Note that this can only occur if $X^{\tau_s-1}=\bm{e}_i$ for some $i$ and $\tau_s\Mod n = i$; this is because for a configuration to drop to $\bm{0}$, the re-randomization operator of that stage of the scan must be equal to the unique nonzero coordinate of the current configuration. Once this happens, each of the subsequent rereandomization operators $i+1,i+2,\ldots$ taken $\Mod n$ each have a $1/2$ chance of moving the configuration back up to the corresponding element $\{\bm{e}_1,\ldots,\bm{e}_n\}$. 

    From this discussion, it is apparent that the law of $\tau_1$ is $n\cdot \mathsf{Geom}(1/2)$; the configuration can only first drop down to $\bm{0}$ at the applications of $\mathsf{P}_n$, each such application does so with probability $1/2$, and one must do the full scan between these trials starting with $X^0$ by our choice of initial configuration. More generally, the law of $\tau_{s+1}-\nu_s$ is an independent $n\cdot \mathsf{Geom}(1/2)$ since once $\nu_s$ occurs as a result of applying the rerandomization of site $\nu_s\Mod n$, the dynamics again must do the full scan before returning to this rerandomization, and these are the only ones that can drop the configuration down to $\bm{0}$. For similar reasons, the laws of $\nu_s-\tau_s$ are independent $\mathsf{Geom}(1/2)$; once the configuration drops to $\bm{0}$, each subsequent rerandomization has probability $1/2$ of causing a transition. It follows that the joint laws of the $\nu_s$ can be expressed as
    \begin{equation}
    \label{eq:law}
        \nu_s = n\sum_{i=1}^s X_i+\sum_{i=1}^{s} Y_i,
    \end{equation}
    where each $X_s,Y_s\sim \mathsf{Geom}(1/2)$ and are all independent. Note that $\mathbb{E}[\nu_s]=2(n+1)s$.

    Let $s^*$ denote the largest index $s$ such that $\nu_s\leq c'n^3$; while this is not a stopping time, it follows that $X^{c'n^3}\in \{\bm{0},\bm{e}_{\nu_{s^*}\Mod n}\}$ depending on whether the configuration managed to drop back down to $\bm{0}$. We claim that for $c'>0$ sufficiently small, there exists a subset $A$ of $[n]$ of size at most $n/4-1$ such that 
    \begin{equation}
    \label{eq:st_concentration}
        \Pr\left(\nu_{s^*}\Mod n\in A\right)\geq 3/4.
    \end{equation}
    If \Cref{eq:st_concentration} holds, this implies the total variation between $c'n^2$ full systematic scans and $\pi$ exceeds $1/2$ since $\pi(\{\bm{e}_i:i\in A\}\cup \{\bm{0}\})\leq 1/4$ while the dynamics has probability at least $3/4$ of being in this set.

    To prove \Cref{eq:st_concentration}, observe that by independence of the geometric random variables in \Cref{eq:law},
    \begin{equation*}
        \mathsf{Var}(\nu_s)=2(n^2+1)s.
    \end{equation*}
    It follows by Chebyshev's inequality that
    \begin{equation*}
    \Pr\left(\vert \nu_s-2(n+1)s\vert \geq 100n\sqrt{s}\right)\leq 1/20.
    \end{equation*}
    Applying this to $s_{\pm}=c'n^2/2\pm \Theta(\sqrt{c'}n)$ shows that
    \begin{equation}
    \label{eq:good_s}
        \Pr(s^*\in [s_-,s_+])\geq 9/10.
    \end{equation}

    To conclude, the Kolmogorov maximal inequality applied to the martingale
    \begin{equation*}
        \sum_{i=1}^s Y_s-2s
    \end{equation*}
    implies that for any $\lambda>0$,
    \begin{equation*}
        \Pr\left(\max_{s\in [0,s_+]} \left\vert \sum_{i=1}^s Y_i-2s\right\vert\geq \lambda\right)\leq \frac{2c'n^2+O(\sqrt{c}n)}{\lambda^2}.
    \end{equation*}
    Taking $\lambda^* = 100\sqrt{c'}n$, we find that
    \begin{equation}
    \label{eq:good_y}
        \Pr\left(\max_{s\in [s_-,s_+]} \left\vert \sum_{i=1}^s Y_i-2s\right\vert\geq \lambda^*\right)\leq 1/10.
    \end{equation}
    Putting together \Cref{eq:good_s} and \Cref{eq:good_y} implies that
    \begin{equation}
    \label{eq:good_A}
        \Pr\left(\sum_{i=1}^{s^*} Y_i \in [2s_--\lambda^*,2s_++\lambda^*]\right)\geq 4/5.
    \end{equation}
    On this event, note that
    \begin{equation*}
        \nu_{s^*}\Mod n = \left(n\sum_{i=1}^{s^*} X_i+\sum_{i=1}^{s^*} Y_i\right)\Mod n=\sum_{i=1}^{s^*} Y_i\Mod n.
    \end{equation*}
    Since $[2s_--\lambda^*,2s_++\lambda^*]$ is a contiguous interval of length $O(\sqrt{c'}n)$, the image $A$ of this set under the $\Mod n$ map also has at most $n/4-1$ elements if $c'$ is a small enough constant, and thus \Cref{eq:good_A} demonstrates $A$ satisfies \Cref{eq:st_concentration}. By the discussion above, this verifies $t^{\mathsf{SS}}_{\mathrm{mix}}(1/4)\geq c'n^2\geq cnt_{\mathrm{mix}}^{\mathsf{GD}}(1/4)$
    so long as $c< c'/C_{\mathsf{GD}}$.
    
\end{proof}

\bibliographystyle{alpha}
\bibliography{bibliography}
\end{document}